\documentclass{article}
\usepackage[utf8]{inputenc}
\usepackage{subfiles}
\usepackage{tikz-cd}
\usepackage{tikz}
\usepackage{bbm}

\usepackage{dsfont}
\usepackage{url}
\usepackage{hyperref}
\usepackage{amsmath}
\usepackage{amssymb}
\usepackage{mathrsfs}
\usepackage{braket}
\usepackage{mathtools}
\usepackage{booktabs}
\usepackage{caption}
\usepackage{cancel}
\usepackage{subfig}
\usepackage{epstopdf}
\usepackage{comment}
\usepackage{appendix}
\DeclarePairedDelimiter{\abs}{\lvert}{\rvert}
\DeclarePairedDelimiter{\norm}{\lVert}{\rVert}
\newcommand{\numberset}{\mathbb}
\newcommand{\N}{\numberset{N}}
\newcommand{\R}{\numberset{R}}

\newcommand{\Z}{\numberset{Z}}
\newcommand{\C}{\numberset{C}}

\newcommand{\sphere}{\mathbb{S}}

\newcommand{\Ham}{\text{Ham}}
\newcommand{\Symp}{\text{Symp}}
\newcommand{\Sym}{\text{Sym}}

\newcommand{\scc}{\mathscr{C}}
\newcommand{\D}{\mathbb{D}}

\newcommand{\Cal}{\text{Cal}}
\newcommand{\Flux}{\text{Flux}}
\newcommand{\supp}{\text{supp}}
\usepackage{amsthm}
\theoremstyle{definition}
\newtheorem{thm}{Theorem}[section]
\newtheorem{lem}[thm]{Lemma}

\newtheorem{cor}[thm]{Corollary}
\newtheorem{defn}{Definition}[section]
\newtheorem{qst}[thm]{Question}
\theoremstyle{remark}
\newtheorem*{rem}{Remark}

\makeatletter
\newcommand*{\@old@slash}{}\let\@old@slash\slash
\def\slash{\relax\ifmmode\delimiter"502F30E\mathopen{}\else\@old@slash\fi}
\makeatother

\makeatletter
\def\bign#1{\mathclose{\hbox{$\left#1\vbox to8.5\p@{}\right.\n@space$}}\mathopen{}}
\def\Bign#1{\mathclose{\hbox{$\left#1\vbox to11.5\p@{}\right.\n@space$}}\mathopen{}}
\def\biggn#1{\mathclose{\hbox{$\left#1\vbox to14.5\p@{}\right.\n@space$}}\mathopen{}}
\def\Biggn#1{\mathclose{\hbox{$\left#1\vbox to17.5\p@{}\right.\n@space$}}\mathopen{}}
\makeatother

\usepackage{graphicx}
\usepackage{wrapfig}
\usepackage[maxbibnames=6, minbibnames=6]{biblatex}
\usepackage{hyperref}
\title{Link Floer Homology and a Hofer 
Pseudometric on Braids}
\author{Francesco Morabito (École Polytechnique, CMLS) \\ email: \href{mailto:francesco.morabito@polytechnique.edu}{francesco.morabito@polytechnique.edu}  }

\begin{document}

\maketitle
\begin{abstract}
    In this paper, we provide new Hofer energy estimates for a class of Hamiltonian diffeomorphisms of the disc. This is formalised following an idea of Frédéric le Roux: we define a family of Hofer-type pseudonorms on braid groups, computing the minimal energy of a Hamiltonian diffeomorphism which fixes a Lagrangian configuration of circles on the unit disc and realises that braid type. We prove that in the case of braids with two strands we have in fact a norm, and we give lower estimates for braids with more strands. The main tool is Link Floer Homology, recently defined by D. Cristofaro-Gardiner, V. Humilière, C.-Y. Mak, S. Seyfaddini and I. Smith, which we use to construct a family of quasimorphisms on $\Ham_c(\D, \omega)$ which is sensitive to the linking number of diffeomorphisms fixing Lagrangian links.
\end{abstract}

\tableofcontents

\section{Introduction}

Let $\D$ be the open unit disc in the complex plane, with its canonical symplectic form $\omega$ normalised so that $\mathrm{Area}(\D)=\int_\D\omega=1$. Given a (possibly time dependent) real-valued function $H\in\scc^\infty(\sphere^1\times \D; \R)$ one can define its Hamiltonian vector field $X_H$ via the equality $\iota_{X_{H_t}}\omega=-dH_t$. Denote by $\Ham(\D, \omega)$ the group of time 1 flows of Hamiltonian vector fields. Since $\D$ is open, we restrict to compactly supported Hamiltonians and their time 1 maps, whose group we note $\Ham_c(\D, \omega)$. With this restriction we ensure that the flow of the Hamiltonian vector field generated by $H$ is defined at every point of the disc and for all times. We denote by $(\phi^t_H)$ this flow, and set $\phi_H:=\phi^1_H$. Recall that $\Ham_c(\D, \omega)$ is a normal subgroup in $\Symp(\D, \omega)$, the group of area-preserving smooth diffeomorphisms of the disc.

One can define a remarkable metric on $\Ham_c(\D, \omega)$, the Hofer metric, as follows. Given a compactly supported Hamiltonian $H$, we define its oscillation to be\[
\norm{H}=\int_0^1\left(\max_{x\in \D}H_t(x)-\min_{x\in \D}H_t(x)\right)\, dt
\]
and we can define the Hofer norm of a compactly supported Hamiltonian diffeomorphism by\[
\norm{\phi}=\inf_{H, \phi=\phi^1_H}\norm{H}.
\]

We define the Hofer distance between two Hamiltonian diffeomorphisms in a bi-invariant way:\[
d_H(\phi, \psi):=\norm{\phi\psi^{-1}}.
\]

Non-degeneracy of the Hofer distance is non trivial to show and was proved in various degrees of generality in \cite{hof90}, \cite{pol93}, \cite{lalMcD95}. A flourishing domain of study is also that of the large scale geometry of the Hofer norm: see for instance \cite{cGHS21}, \cite{polShel23} and references therein.
\begin{rem}It is in general complicated to find Hamiltonian diffeomorphisms with large Hofer norm. Known examples yield the following heuristics: if a diffeomorphism ``moves around big open sets in a complicated manner'', it should have large Hofer norm. The main result of this paper should be seen as a confirmation of this general principle in the case of the disc.
\end{rem}
One of the main tools for the study of the Hofer metric on $\Ham$ are quasimorphisms, introduced in the field of Symplectic Topology in \cite{entPol03}. A \textbf{quasimorphism} on a group $G$ (in our case $G=\Ham_c(\D, \omega)$) is a function $Q: G\rightarrow\R$ such that there is a constant $D\geq 0$, the defect of the quasimorphism, verifying for all $g, h\in G$\[
\abs{Q(gh)-Q(g)-Q(h)}\leq D.
\]
A quasimorphism is said to be \textbf{homogeneous} if it is a homomorphism when restricted to powers of the same element: $\forall g\in G, \forall k\in \Z$, $Q(g^k)=kQ(g)$. Homogeneous quasimorphisms may be used to study properties of Hamiltonian diffeomorphism groups with respect to the Hofer metric when they are Hofer-Lipschitz. The research for and the study of Hofer-Lipschitz quasimorphisms has now a long history and a wide plethora of applications, some of which can for instance be found in \cite{Cghmss21} \cite{entPol03}, \cite{EPP08}, \cite{kaw22} \cite{kha09}, \cite{py06}, and several others. The research for quasimorphisms on $\Ham(M, \omega)$ may be justified by the simplicity theorem of Banyaga \cite{ban78} (see also \cite{ban97} and \cite{McDuffSalamon94}): for a closed manifold $\Ham$ is simple, so that it does not admit any non-trivial homomorphism (kernels are normal subgroups). If however $M^{2n}$ is an open manifold with symplectic form $\omega=d\lambda$ (for us $M=\D$ with its standard structure) there's a natural homomorphism on $\Ham_c(M, \omega)$:
\[
\text{Cal}:\Ham_c(M, \omega)\rightarrow\D,\,\, \phi_H\rightarrow \int_0^1\int_MH_t\omega^n dt
\]
$\mathrm{Cal}$ will later appear in this paper as a part of definitions of new quasimorphisms.

\subsection{Braid Groups}
Braid groups are classical objects in group theory. Typical references on the subject include the book \cite{kasTur08} and the paper \cite{gon11}. The name of the group was coined by Emil Artin in \cite{art25}, where he gave an abstract definition in terms of generators and relations: for $k\geq 2$,
\begin{equation*}
\mathcal{B}_k = \left \langle \sigma_1, \dots, \sigma_{k-1} \middle\vert
\begin{aligned}
 \sigma_i\sigma_j & = \sigma_j\sigma_i : |i-j| > 1 \\
 \sigma_i\sigma_{i+1} \sigma_i & = \sigma_{i+1} \sigma_i\sigma_{i+1} : 1 \leqslant i\leqslant k-2
\end{aligned}
\right\rangle.
\end{equation*}
As explained in the references above however, there are more geometrical ways of seeing the braid groups, the one we are going to use in the present paper being the following one. In the cartesian product $\D^k$, let $\Delta$ be the fat diagonal (see Section \ref{section:descrTheory}). The permutation group on $k$ letters $\mathfrak{S}_k$ admits a faithful action on $\D^k\setminus\Delta$, and we call the quotient
\[
\mathrm{Conf}^k(\D):=(\D^k\setminus\Delta)/\mathfrak{S}_k
\]
the configuration space of $k$ unordered points of the disc. It turns out that $\pi_1(\text{Conf
}^k(\D))=\mathcal{B}_k$. This means essentially that after choosing $k$ base points $p_1, \dots, p_k$ on $\D$, one can define an element of $\mathcal{B}_k$ uniquely as the homotopy class of a path $\gamma: [0, 1]\rightarrow\D^k\setminus \Delta$ such that there is $\sigma\in\mathfrak{S}_k$ verifying $\gamma(1)=\sigma \gamma(0)$, and that any braid may be realised this way. The multiplication in $\mathcal{B}_k$ then corresponds to concatenation of braids the obvious way.

There is a quotient homomorphism $\mathcal{B}_k\rightarrow \mathfrak{S}_k$, mapping positive and negative generators to the same transposition, or equivalently identifying $\sigma_i$ with $\sigma_i^{-1}$. Its kernel is called $P_k$, the set of pure braids. One can visualise pure braids as braids such that, following the strands, bring each basepoint on itself.

In the following sections of this work we shall need the definition of linking number\footnote{In the low-dimensional topology literature it may also be called ``exponent''.} \[
\mathrm{lk}: \mathcal{B}_k\rightarrow\Z.
\]
The function $\mathrm{lk}$ is defined as the only morphism of groups whose value on all the positive generators is 1. Note that this convention, very natural from the algebraic viewpoint, requires a different normalisation from the usual one: choosing basepoints $0$ and $\frac{1}{2}$ in $\D$, the braid represented by $t\mapsto (0, \frac{1}{2}\exp\left(2\pi i t\right))$ has linking number 2 according to our definition, while it is customarily used as an example of a loop with a linking number 1 in geometric settings.

\subsection{The Main Results}

In this paper we are going to find a lower bound on the Hofer norm of Hamiltonian diffeomorphisms which preserve certain configurations of circles, given a ``braid type'' they realise.

\begin{defn}
A family of embedded circles $L_i, i=1,\dots, k$ in $\D$ is called \textbf{pre-monotone} if the circles are disjoint, if there are discs $D_i, i=1,\dots, k$ and $A\in\left(\frac{1}{k+1}, \frac{1}{k}\right)$ such that $L_i=\partial D_i$ and $\mathrm{Area}(D_i)=A$.
\end{defn}

Let $\underline{L}$ indicate the choice of an unordered pre-monotone configuration of circles. Define the subgroup $\Ham_{\underline{L}}(\D, \omega)$  of $\Ham_c(\phi, \omega)$ by
\[
\Ham_{\underline{L}}(\D, \omega)=\Set{\phi\in\Ham_c(\D, \omega)\vert \phi(\underline{L})=\underline{L}}.
\]

We may associate a braid type to every diffeomorphism in $\Ham_{\underline{L}}(\D, \omega)$ the following way. Choose base points $p_i\in L_i$, and for any $\phi\in \Ham_{\underline{L}}(\D, \omega)$ consider a Hamiltonian isotopy between $\mathrm{Id}$ and $\phi$, $(\phi^t)$. This gives a path $(\phi^t(p_i))$ in $\mathrm{Conf}^k(\D)$ starting and ending on $\underline{L}$, therefore for all $i$ there exists a $\sigma(i)$ with $\phi(p_i)\in L_{\sigma(i)}$. Joining $\phi(p_i)$ with $p_{\sigma(i)}$ using a path contained in $L_{\sigma(i)}$ one finds a braid which we denote by $b(\phi, \underline{L})$. Its braid type (i.e. equivalence class modulo braid isotopy) does not depend on the path we choose since $L_i\cap L_j=\emptyset$ for $i\neq j$, nor on the Hamiltonian isotopy we use, as $\pi_1(\Ham_c(\D, \omega))=0$. Given a pre-monotone Lagrangian link $\underline{L}$ of $k$ circles and a braid $b\in \mathcal{B}_k$ we may define the subset\[
\Ham_{b, \underline{L}}(\D, \omega)=\Set{\phi\in \Ham_{\underline{L}}(\D, \omega)\vert b(\phi, \underline{L})=b}.
\]

\begin{rem}
The $\Ham_{b, \underline{L}}(\D, \omega)$ are clearly not subgroups of $\Ham_c(\D, \omega)$, due to the morphism property of the map $\phi\mapsto b(\phi, \underline{L})$: if $\phi, \psi\in \Ham_{\underline{L}}(\D, \omega)$,\[
b(\phi\circ\psi, \underline{L})=b(\psi, \underline{L})\#b(\psi,\underline{L}).
\]
\end{rem}

\begin{rem}
We choose to require in this definition that the circles bound discs of the same area. Allowing for discs of different areas might result in a slight generalisation, but since we are considering Hamiltonian maps (which in particular preserve areas), if the configuration is not pre-monotone we are not going to be able to produce any braid in $\mathcal{B}_k$ as a braid type of a Hamiltonian diffeomorphism. If the configuration is pre-monotone however the morphism $\Ham_{\underline{L}}(\D, \omega)\rightarrow \mathcal{B}_k$ is easily seen to be surjective: one can realise all the generators of $\mathcal{B}_k$. See Section \ref{section:B2} where we do it for $k=2$. The case $k\geq 3$ is totally analogous.
\end{rem}

Using this notation, we can define a family of Hofer pseudonorms on braid groups. Given a pre-monotone Lagrangian configuration $\underline{L}$, following an idea of Frédéric Le Roux we define
\[
\norm{\cdot}_{\underline{L}}: \mathcal{B}_k\rightarrow\R, \, \norm{b}_{\underline{L}}=\inf\Set{\norm{\phi}\vert \phi\in \Ham_{b, \underline{L}}(\D, \omega)}.
\]

The following question is whether $\norm{\cdot}_{\underline{L}}$ is a norm. It is easy to see that $\norm{\cdot}_{\underline{L}}$ is at least a pseudonorm:

\begin{lem}
    For all $g_1, g_2\in\mathcal{B}_k$, $\norm{g_1}=\norm{g^{-1}_1}$ and $\norm{g_1g_2}\leq \norm{g_1}+\norm{g_2}$.
\end{lem}
\begin{proof}
    The first point is obvious since for all $\phi\in \Ham_c(\D, \omega)$ $\norm{\phi}=\norm{\phi^{-1}}$ and $b(\phi, \underline{L})=b(\phi^{-1}, \underline{L})^{-1}$. For the second one,\[
    \norm{g_1g_2}_{\underline{L}}=\inf_{\phi\in\Ham_{\underline{L}}(\D, \omega), b(\phi, \underline{L})=g_1g_2}\norm{\phi}\leq \inf_{\phi_i, \in\Ham_{\underline{L}}(\D, \omega), b(\phi_i, \underline{L})=g_i}\norm{\phi_1\phi_2}
    \]
    and applying triangular inequality
    \[
    \norm{g_1g_2}_{\underline{L}}\leq \norm{\phi_1}+\norm{\phi_2},\,\, \forall \, \phi_i\in \Ham_{\underline{L}}(\D, \omega)\text{ with }b(\phi_i, \underline{L})=g_i
    \]
    so that taking the infimum first on $\phi_1$ and then on $\phi_2$ we conclude.
\end{proof}

Our main result is a kind of braid persistence which does not depend on the action of periodic orbits. Define a pseudodistance on $\mathcal{B}_k$ by the formula $d_{\underline{L}}(g, h)=\norm{gh^{-1}}_{\underline{L}}$. Our main result is the following:
\begin{thm}\label{thm:persLink}
    If $g, h\in\mathcal{B}_k$, then\[
    d_{\underline{L}}(g, h)\geq \frac{(k+1)A}{3(k+1)A+1}\frac{1}{k}\frac{(k+1)A-1}{2(k-1)}\abs{\mathrm{lk}(gh^{-1})}.
    \]
\end{thm}
\vspace{0.2 cm}
It is interesting to observe that as $A\to\frac{1}{k+1}$ the lower bound becomes 0, and this coincides with the area value for which the discs become displaceable in the complement of the others. The best estimate we can provide instead appears as $A\to\frac{1}{k}$, as could be expected.
\vspace{0.2 cm}

A corollary of Theorem \ref{thm:persLink} is about non-degeneracy of Hofer norms for braids with 2 strands.
\begin{cor}\label{thm:nondeg}
For any pre-monotone configuration $\underline{L}$ with two components of areas $A\in\left(\frac{1}{3}, \frac{1}{2}\right)$, the Hofer pseudonorm $\norm{\cdot}_{\underline{L}}$ on $\mathcal{B}_2\cong \Z$  is non degenerate.
\end{cor}

Our results do not solve this problem in full generality due to the fact that there exist braids with three or more strands which have 0 linking number but are not trivial.

\begin{rem}Since the appearance of this paper, G. Chen has proved non-degeneracy of the Hofer norms on braid groups with arbitrary number of strands in his work \cite{chen23}. The lower bounds he finds for the Hofer norm of non trivial braids do not depend however on the complexity of the braid itself. It would be interesting to find explicit lower bounds based on a measure of complexity of the braid which allow for a proof of non-degeneracy of these norms.
\end{rem}

Theorem \ref{thm:persLink} itself follows from the construction of a quasimorphism using Link Floer Homology, described in Section \ref{section:descrTheory} following the paper \cite{Cghmss21}.

\begin{thm}\label{thm:qmorph}
Let $\underline{L}$ be a pre-monotone Lagrangian link with $k$ components of area $A\in \left(\frac{1}{k+1}, \frac{1}{k}\right)$. There exists a Hofer $\left(3+\frac{1}{(k+1)A}\right)$-Lipschitz homogeneous quasimorphism $Q_k: \Ham_c(\D, \omega)\rightarrow\R$ such that if $\phi\in \Ham_{\underline{L}}(\D, \omega)$ has associated braid type $b(\phi, \underline{L})$, then
\[
Q_k(\phi)=\frac{1}{k}\frac{(k+1)A-1}{2(k-1)}\mathrm{lk}(b(\phi, \underline{L})).
\]
\end{thm}

The proof of this Theorem will be given in Section \ref{section:proofMain}. A first step in the direction of this Theorem is given in \cite{kha11}, where the author constructs a family of quasimorphisms obtained from those defined in \cite{entPol03} to prove a weaker result when $\underline{L}$ has two components, one of them being fixed throughout the isotopy. In his case however there are no hypotheses made on the area bounded by the fixed component.

\begin{rem}
    We can refine the result from Theorem \ref{thm:qmorph} to non-homogenised spectral invariants on $\Ham_c(\D, \omega)$, the proof being identical.

    Moreover, there exists in fact a linear independent family of quasimorphisms with the feature that restricted to $\Ham_{\underline{L}}(\D, \omega)$ they are proportional to the linking of the braid associated to the diffeomorphism; the quasimorphism in Theorem \ref{thm:qmorph} happens to be the one with the largest proportionality constant.
\end{rem}

We end this section remarking that Theorem \ref{thm:persLink} gives for free a wide class of Hamiltonian diffeomorphisms which are in the kernel of Calabi and with asymptotic Hofer norm (which  we shall define in a moment) bounded away from 0.

Given a Hamiltonian diffeomorphism $\phi\in \Ham_c(\D, \omega)$, its asymptotic Hofer norm is defined as\begin{equation*}
    \norm{\phi}_\infty:=\lim_{n\to\infty}\frac{1}{n}\norm{\phi^n}.
\end{equation*}
Since $\Cal$ is a Hofer-Lipshitz homomorphism, it is trivial to see that every $\phi$ not in the kernel of Calabi has non-zero asymptotic Hofer norm. Thanks to Theorem \ref{thm:persLink} we can provide several examples of compactly supported Hamiltonian diffeomorphisms of the disc in the kernel of Calabi which have arbitrarily big asymptotic Hofer norm. Take any $\phi\in \Ham_{\underline{L}}(\D, \omega)$. Since the function
\begin{equation*}
    b(\cdot, \underline L): \Ham_{\underline{L}}(\D, \omega)\rightarrow \mathcal{B}_k
\end{equation*}
is a group homomorphism, we find
\begin{equation*}
    \mathrm{lk} \, b(\phi^n, \underline{L})=\mathrm{lk} \,(b(\phi, \underline{L})^n)=n\,\mathrm{lk} \,b(\phi, \underline{L}).
\end{equation*}
Consider now any $\psi\in \Ham_{\underline{L}}(\D, \omega)$ such that\begin{itemize}
    \item $\psi$ is supported in one of the discs bounded by a circle in $\underline{L}$ (this implies that $b(\psi, \underline L)$ is trivial);
    \item $\Cal(\psi)=-\Cal(\phi)$.
\end{itemize}

The diffeomorphism $\phi\circ \psi$ is the diffeomorphism of the kind we wanted to construct: $\Cal(\phi\circ \psi)=0$ and $b(\phi\circ\psi, \underline L)=b(\phi, \underline L)$, so that
\begin{equation*}
        \norm{\phi\circ \psi}_\infty\geq \frac{(k+1)A}{3(k+1)A+1}\frac{1}{k}\frac{(k+1)A-1}{2(k-1)}\abs{\mathrm{lk} \, b(\phi, \underline L)}.
    \end{equation*}

\subsection{Further directions}
The questions we have started discussing in this paper may be easily generalised to other, new cases.

\begin{qst}
    Given a quasimorphism $\tilde{Q}_k$ on the braid group $\mathcal{B}_k$, does there exist a Hofer Lipschitz, $\scc^0$-continuous quasimorphism  $Q_k$ on $\Ham_c(\D, \omega)$ such that, for a Lagrangian configuration $\underline{L}$ with $k$ components, the restriction of $Q_k$ to $\Ham_{\underline{L}}(\D, \omega)$ coincides with $\tilde{Q}_k$ up to a constant?
\end{qst}

We are here interpreting quasimorphisms on braid groups as measures of complexity of braids. Answering this question would be helpful to determine lower bounds on Hofer norms on braid groups, and to examine new braid-persistence phenomena.

\begin{qst}
    Is it possible to say something when the Lagrangian configuration is not pre-monotone?
\end{qst}
This happens, for instance, if one of the enclosed discs is displaceable in the complement of the others. In general if the circles bound discs of different areas they cannot be arbitrarily permuted using a Hamiltonian flow, if the resulting braid is not pure. The result by Khanevsky in \cite{kha11} makes again the first step in this direction. We believe that his proof may be generalised using our methods, but the result would be an estimate of Hofer norm based on relative linking with respect to one single disc of different area which does not move throughout the Hamiltonian isotopy. More work is needed to answer the question in more generality.

\begin{qst}
    Is it possible to define an analogous family of pseudodistances on  other surfaces, possibly with genus? Would it still be unbounded?
\end{qst}
Another work of Khanevsky's, \cite{kha15}, paves the way and shows how in the presence of genus things might become more unpredictable. He defines a notion of homological trajectory for Hamiltonian diffeomorphisms preserving a chosen disc, and proceeds to prove that if the genus is positive then one can realise arbitrarily complicated trajectories with bounded Hofer energy.

\begin{qst}
    Can $\norm{\cdot}_{\underline{L}}$ be used to approximate the Hofer norm of an arbitrary diffeomorphism on the disc?
\end{qst}

As it is remarked in \cite{Cghmss21}, if $\phi\in\Ham_c(\Sigma, \omega)$ there is a Lagrangian link with $k$ components such that $\phi$ can be $\varepsilon$-Hofer deformed to $\tilde{\phi}\in\Ham_{\underline{L}}(\Sigma, \omega)$. This means that if $\mathcal{L}$ is the family of all possible Lagrangian links on $\Sigma$,\[
\Ham_c(\Sigma, \omega)=\overline{\bigcup_{\underline{L}\in\mathcal{L}}\Ham_{\underline{L}}(\Sigma, \omega)}
\]
where the closure is taken with respect to the Hofer topology. The same remark however shows that for extremely fine links (as $k\to\infty$) the obtained braid type is trivial, so the Hofer pseudonorm on braid groups will not provide any meaningful information. Moreover, as our estimates rapidly  degrade for increasing $k$, it would be interesting to find an algorithm that, up to a $\varepsilon$-Hofer perturbation, finds a Lagrangian link respected by the diffeomorphism with the fewest components possible, and such that the associated braid is non-trivial. We believe the proof in Section \ref{section:B2} might be of help.
\vspace{0.5cm}

This paper is organised as follows. In Section \ref{section:descrTheory} we give a description of Link Floer Homology following \cite{Cghmss21},  in Section \ref{section:proofMain} we prove the main result. Section \ref{section:B2} contains a computation of the value of $Q_2$ when applied to braids with two strands (the strategy here is due to M. Khanevsky \cite{kha11}). The result in Section \ref{section:B2} is of course implied by the one in Section \ref{section:proofMain}, but we still include this proof separately as it does not require anything more than basic properties of the quasimorphisms constructed in Section \ref{section:descrTheory}.

\textit{Acknowledgments}: The author wishes to thank his PhD supervisor, Vincent Humilière, for the constant support and help provided, and Sobhan Seyfaddini for his questions. Thanks are also due to Julien Grivaux, for his effort in explaining properties of holomorphic functions between symmetric products, and to an anonymous referee for their patient job which greatly improved the paper.

\section{Link Floer Homology}\label{section:descrTheory}
The tool we use in the proof of our main result is the family of quasimorphisms defined in \cite{Cghmss21}. Let us now sketch their construction, referring to that paper for details. These quasimorphisms come as homogenisation of spectral invariants of a Lagrangian Floer theory associated to a collection of curves on a surface.

To fix notation, if $\Sigma$ is a closed surface, $Sym^k(\Sigma):=\Sigma^k/\mathfrak{S}_k$ denotes its $k$-fold symmetric product, and $\Delta$ is its singular locus, i.e. the fat diagonal
\[
\Delta=\Set{(z_1, \dots, z_k)\vert \exists i\neq j, z_i=z_j}/\mathfrak{S}_k.
\]
The action of $\mathfrak{S}_k$ is given by permuting the factors: for $(z_1, \dots, z_k)\in \Sigma^k$ and $\sigma\in \mathfrak{S}_k$,
\[
\sigma\cdot(z_1,\dots, z_k):=(z_{\sigma(1)}, \dots, z_{\sigma(k)}).
\]
If $\Sigma$ is also oriented, it is symplectic and the symplectic form $\omega^{\oplus k}$ descends to a singular symplectic form $\omega_{Sym^k(\Sigma)}$ on the symmetric product. While the symmetric product of manifolds in full generality is not a manifold but an orbifold, one can give a structure of complex manifold to symmetric products of Riemann surfaces. In particular $Sym^k(\C P^1)$ is biholomorphic to $\C P^k$. The symplectic structure on the quotient can be smoothed out using a procedure devised by Perutz following a paper by Varouchas: for details we refer to \cite[Section 7]{per05} and \cite{Cghmss21}. The fact that we use a Perutz-type symplectic form on the symmetric product will be in the background without being explicitly mentioned, but it is fundamental in the construction of the homology theory we are about to describe carried out in \cite{Cghmss21}.

In what follows, $X:=Sym^k(\sphere^2)$, and $\omega_X$ is the  singular symplectic form induced by the Fubini-Study metric.

Let $L_1\times \dots\times L_k=:\underline{L}\subset X$ be a collection of $k$ non-intersecting circles on the sphere (i.e. $\underline{L}\cap \Delta=\emptyset$). Let $(B_i)$ be the collection of connected components of $\sphere^2\setminus \bigcup_{i=1}^kL_i$, $k_i$ be the number of connected components of $\partial B_i$, and $A_i=Area(B_i)$. The collection $\underline{L}$ is said to be $\eta$-\textbf{monotone}, for some $\eta\geq 0$, if there exists a $\lambda>0$ such that for every $i$
\begin{equation}\label{eq:monotonicity}
\lambda = 2\eta(k_i-1)+A_i. 
\end{equation}

We remark that $Sym^k(\underline{L})$ is in fact a Lagrangian in the symmetric product, and the authors of \cite{Cghmss21} proceed to the definition of a Lagrangian Floer complex associated to any Hamiltonian deformation of it. More precisely, let $H:\sphere^2\times \sphere^1\rightarrow \R$ be a periodic Hamiltonian of time 1 map $\phi=\phi^1_H$. It naturally defines a Hamiltonian function $Sym^k(H)$ and a (non-smooth) Hamiltonian diffeomorphism of the symmetric product which we denote $Sym^k(\phi_H)$. Assume that for all $i\neq j\in \{1, \dots, k\}$ we have $\phi(L_i)\pitchfork L_j$, and let $\textbf{x}\in Sym^k(\phi_H)(Sym^k(\underline{L}))\cap Sym^k(\underline{L})$. Let us define the path $\boldsymbol\alpha(t)=\phi^{1-t}_H(\textbf{x})$, and consider $\tilde{\mathcal{S}}$, the set of all intersection points $Sym^k(\underline{L})\cap Sym^k(\phi_H)(Sym^k(\underline{L}))$ which are homotopic to $\boldsymbol{\alpha}$ through paths between $Sym^k(\phi_H)(Sym^k(\underline{L}))$ and $Sym^k(\underline{L})$, where we think of elements in $Sym^k(\phi_H)(Sym^k(\underline{L}))\cap  Sym^k(\underline{L})$ as constant paths from $Sym^k(\phi_H)(Sym^k(\underline{L}))$ to $Sym^k(\underline{L})$. A homotopy $\hat{y}:([0,1]\times[0,1], [0, 1]\times \{0, 1\})\rightarrow (X, Sym^k(\underline{L}))$ between an intersection point $y\in Sym^k(\phi_H)(Sym^k(\underline{L}))\cap Sym^k(\underline{L})$ and $\boldsymbol{\alpha}$ is called a \textbf{capping}. By \cite{Cghmss21}, Lemma 4.10, the image of the Hurewicz morphism $\pi_2(X, Sym^k(\underline{L}))\rightarrow H_2(X, Sym^k(\underline{L}))$ is freely generated by $s$ homology classes, $u_1, \dots, u_s$, where $s$ is the number of connected components in $\sphere^2\setminus \bigcup L_i$. In the case we are interested about there are $k+1$ connected components and all of them, except for one, are discs. In this case we have $k+1$ homology classes, $u_1, \dots, u_{k+1}$, and we get the refined information about their intersections with the diagonal:\[
u_i\cdot \Delta=2(k_i-1)
\]
where $k_i=1$ for $i=1, \dots, k$ and $k_{k+1}=k$.

Each of these homology classes may now be used to change capping of an orbit using the following method. If $u$ represents any of the $u_i$ above and $\hat{y}$ is a capped orbit, we may consider their concatenation $\hat{y}u=\hat{y}\#(\phi^{1-t}_Hu(s,t))$, and the quantity $\omega_X(u)+u\cdot \Delta$ does not depend on the choice of representative $u$. It makes sense then to define the set of capped orbits modulo equivalence
\[
\mathcal{S}=\Set{\hat{y}\vert y\in\tilde{\mathcal{S}}}/\sim
\]
where $\hat{x}\sim\hat{y}\Leftrightarrow x=y, \omega_X(\hat{x})+\eta\hat{x}\cdot\Delta=\omega_X(\hat{y})+\eta\hat{y}\cdot\Delta$. Define now\[
\widetilde{CF}^\bullet( Sym^k(\phi_H), \Sym^k(\underline{L}); \C)=\bigoplus_{\hat{y}\in \mathcal{S}}\C\cdot[\hat{y}]
\]

Recapping gives a $\Z$ action on $\widetilde{CF}^\bullet(Sym^k(\phi_H), \Sym^k(\underline{L}); \C)$, so that we can see it as a $\C[T, T^{-1}]$-module. The Floer complex we are interested in is the tensor product
\[
CF^\bullet(\phi_H, \underline L; \C)=\widetilde{CF}^\bullet(Sym^k(\phi_H), \Sym^k(\underline{L}); \C)\otimes_{\C[T, T^{-1}]}\Lambda
\]
with the Novikov field $\Lambda$
\[
\Lambda=\C[[T]][T^{-1}]=\Set{\sum_{i=0}^\infty a_iT^{b_i}\vert a_i\in \C, b_i\in\Z, b_i<b_{i+1}}.
\]
The differential on $CF^\bullet(\phi_H, \underline L; \C)$ as customary is defined by a count of holomorphic curves with Lagrangian boundary conditions: fix two intersections $y_i\in Sym^k(\phi_H)(Sym^k(\underline{L}))\cap Sym^k(\underline{L}), i=0, 1$ and a time-dependent almost complex structure on $X$; a $J-$holomorphic strip between $y_0$ and $y_1$ is a map $u: \R\times [0, 1]\rightarrow X$ satisfying
\begin{equation}
 \begin{cases}
 u(s, 0)\in Sym^k(\phi_H)(Sym^k(\underline{L})),\,\, u(s, 1)\in Sym^k(\underline{L})\\
 \lim_{s\to -\infty}u(s, t)=y_0, \,\,\lim_{s\to +\infty}u(s, t)=y_1\\
 (\partial_s+J_t\partial_t)u(s, t)=0
 \end{cases}  . 
\end{equation}
\begin{rem}\label{rem:natACS}
    Given a holomorphic structure $J$ on $\sphere^2$, one naturally induces a holomorphic structure on $X$, denoted $J_X$.
\end{rem}
Solutions to these equations belonging in moduli spaces of virtual dimension 0 (after holomorphic reparametrisation of $\R+i[0, 1]$) are regular, so that a differential can be well defined, and it can be proved that $\partial^2=0$\footnote{In \cite{Cghmss21} the authors define the complex with coefficients in a local system, and then prove that the differential squares to 0 using the trivial local system. In view of this, we just use complex coefficients for simplicity.}. Along these curves
\begin{equation}\label{eq:actionDefinition}
\mathcal{A}^\eta_H( [\hat{y}])=\int_0^1 Sym^k(H)(\textbf{x})\, dt-\int_\D \hat{y}^*\omega_X-\eta[\hat{y}]\cdot\Delta
\end{equation}
strictly decreases. The quantity $\mathcal{A}^\eta_H$ is called the \textbf{action}, and $CF^\bullet(\phi_H, \underline{L}; \C)$ is a filtered differential complex. We remark that to make sense of this, we need to observe that $\mathcal{A}^\eta_H$ is constant on the equivalence classes in $\mathcal{S}$, and to make the interaction between the $\Z$-action (or equivalently, the action of the formal variable $T$) and $\mathcal{A}^\eta_H$ explicit. As it is now classical (see for instance \cite{UshZh16}) it corresponds to a translation, and given our monotonicity requirement\[
\mathcal{A}^\eta_H( [\hat{y}]T)=\mathcal{A}^\eta_H([\hat{y}])-\lambda.
\]

Let $HF^\bullet_a(\phi_H, \underline{L}; \C)$ be the homology of the subcomplex
\[
CF^\bullet_a(\phi_H, \underline{L}; \C)=\bigoplus_{[\hat{y}]\in \mathcal{S},\,\, \mathcal{A}^\eta_H([\hat{y}])<a}\C\cdot[\hat{y}].
\]

In analogy with classical theories, we define the set $k\mathrm{Spec}(H:\underline{L})$ as the action values of capped intersection points. It is a closed and nowhere dense subset of $\R$ (\cite{Oh05}). The notation implies that elements in the \textbf{spectrum} $\text{Spec}(H:\underline{L})$ are action values of generators divided by $k$ (as in \cite[Definition 6.2]{Cghmss21}).

\begin{rem}
The definition of the action functional in Equation \ref{eq:actionDefinition} is an extension of the usual one: the term counting intersections with the diagonal depends on a procedure of inflation of a Perutz-type symplectic form one does in order to achieve monotonicity of the Lagrangian link in the symmetric product, when the circles satisfy the larger definition of $\eta$-monotonicity. For details, we refer to \cite[Remark 4.22]{Cghmss21}.
For the following it is necessary to notice that one makes no normalisation assumptions on the symplectic form in order to define the action as in \ref{eq:actionDefinition}.
\end{rem}

This Lagrangian Floer theory comes with its PSS isomorphisms (\cite{Cghmss21}, \cite{Zap15}), and as $\Lambda$-vector spaces $HF^\bullet(\phi_H, \underline{L}; \C)\cong H((\sphere^1)^k; \Lambda)$. Therefore $HF^\bullet(\phi_H, \underline{L}; \C)$ is non trivial, and it moreover has a multiplicative structure (\cite{ChoOh}), so let $e\in HF^\bullet(\phi_H, \underline{L}; \C)$ be the unit for its product and define\begin{align*}
c_{Sym^k(\underline{L})}(H)=\inf\{a\in \R\vert  e\in \mathrm{Im}(HF_a^\bullet(\phi_H, \underline{L}; \C)\\ \rightarrow HF^\bullet(\phi_H, \underline{L}; \C))\}
\end{align*}
and $c_{\underline{L}}:=\frac{1}{k}c_{Sym^k(\underline{L})}$. One of the fundamental results in \cite{Cghmss21} is the following Theorem:
\begin{thm}[\cite{Cghmss21}, Theorem 1.13]\label{prop:c_L}
For any $\eta$-monotone Lagrangian link $\underline{L}$ on $\Sigma$, the link spectral invariant\[
c_{\underline{L}}: \scc^\infty([0, 1]\times \Sigma)\rightarrow \R
\]
satisfies the following properties:
\begin{itemize}
    \item (Spectrality) For any $H$, $c_{\underline{L}}(H)\in \text{Spec}(H: L)$.
    \item (Hofer Lipschitz) For any $H, H'$ Hamiltonians,\[
    \int_0^1\min_{x\in \Sigma}(H_t(x)-H'_t(x))\, dt \leq c_{\underline{L}}(H)-c_{\underline{L}}(H')\leq \int_0^1\max_{x\in\Sigma}(H_t(x)-H'_t(x))\, dt.
    \]
    \item (Lagrangian Control) If for each  $i\in\{1,\dots, k\}$, we have $H_t\vert_{L_i}=s_i(t)$ for time-dependent constant $s_i:[0, 1]\rightarrow \R$, then\[
    c_{\underline{L}}(H)=\frac{1}{k}\sum_{i=1}^k\int_0^1s_i(t) dt
    \]
    and for a general Hamiltonian
    \[
    \frac{1}{k}\sum_{i=1}^k\int_0^1 \min_{x\in L_i}H_t(x)\, dt\leq c_{\underline{L}}(H)\leq \frac{1}{k}\sum_{i=1}^k\int_0^1\max_{x\in L_i}H_t(x)\, dt.
    \]
    \item (Subadditivity) For any $H, H'$, $c_{\underline{L}}(H\#H')\leq c_{\underline{L}}(H)+c_{\underline{L}}(H')$, if $H\#H'(x, t)=H_t(x)+H'_t((\phi^t_H)^{-1}(x))$.
    \item (Homotopy Invariance) If $H, H'$ are two normalised Hamiltonians with same time 1 map and which determine the same element in $\widetilde{\Ham}(\Sigma, \omega)$, then $c_{\underline{L}}(H)=c_{\underline{L}}(H')$.
    \item (Shift) If $H=H'+s$, for a function $s\in \scc^\infty([0, 1]; \R)$, then\[
    c_{\underline{L}}(H)=c_{\underline{L}}(H')+\int_0^1s(t)\, dt.
    \]
    \end{itemize}
\end{thm}

The proof of this Theorem can be found in Section 6.4 of \cite{Cghmss21}.

\vspace{0.3cm}
Thanks to the property (Homotopy Invariance), $c_{\underline L}$ lifts to the universal cover $\widetilde{\Ham}(\Sigma, \omega)$ of the Hamiltonian group. If $\tilde\phi$ is represented by a normalised Hamiltonian $H$, we set
\begin{equation*}
c_{\underline{L}}(\tilde\phi):=c_{\underline L}(H).
\end{equation*}
In the case of a sphere, we may homogenise the spectral invariants $c_{\underline{L}}$ to find quasimorphisms on $\Ham(\sphere^2, \omega)$: if $\tilde{\phi}\in\widetilde{\Ham}(\sphere^2, \omega)$ let
\[
\mu_{\underline{L}}(\tilde{\phi})=\lim_{n\to\infty}\frac{c_{\underline{L}}(\tilde{\phi}^n)}{n},
\]
and since $\pi_1(\Ham(\sphere^2, \omega))$ is finite $\mu_{\underline{L}}$ only depends on the time 1 map. We can then formulate the following theorem, proved in Section 7 of \cite{Cghmss21}.
\begin{thm}[\cite{Cghmss21}, Theorems 7.6, 7.7]\label{prop:mu_k}
$\mu_{\underline{L}}:\Ham(\sphere^2)\rightarrow \R$ satisfies the following properties:
\begin{itemize}
    \item $\mu_{\underline{L}}$ only depends on the monotonicity constant $\eta$ and on the number of components of the Lagrangian link $\underline{L}$. We will then write $\mu_{k, \eta}$ for the quasimorphism associated to any $\eta$-monotone Lagrangian link of $k$ components.
    \item If $k\neq k'$ or $\eta\neq \eta'$, then $\mu_{k, \eta}$ and $\mu_{k', \eta'}$ are linearly independent.
    \item The differences $\mu_{k, \eta}-\mu_{k', \eta'}$ are $\scc^0$-continuous.
\end{itemize}
Moreover, the properties in Theorem \ref{prop:c_L} translate to the following ones for the $\mu_{k, \eta}$:
\begin{itemize}
    \item (Hofer Lipschitz) $\abs{\mu_{k, \eta}(\phi)-\mu_{k, \eta}(\psi)}\leq d_H(\phi, \psi)$.
    \item (Lagrangian Control) If $\phi=\phi^1_H$ for a mean normalised Hamiltonian $H$ such that $H_t\vert_{L_i}=s_i(t)$,\[
    \mu_{k, \eta}(\phi)=\frac{1}{k}\sum_{i=1}^k\int_0^1s_i(t) \, dt.
    \]
    In general, if $H$ is mean-normalised but not necessarily a time dependent constant on the Lagrangian link, then:
    \[
    \frac{1}{k}\sum_{i=1}^k\int_0^1 \min_{x\in L_i}H_t(x)\, dt\leq \mu_{k, \eta}(\phi)\leq \frac{1}{k}\sum_{i=1}^k\int_0^1\max_{x\in L_i}H_t(x)\, dt.
    \]
    \item (Support Control) If $\phi=\phi^1_H$ where $\text{Supp}(H)\subset \sphere^2\setminus \bigcup_i L_i$, then $\mu_{k, \eta}(\phi)=-\Cal(\phi)$.
\end{itemize}
\end{thm}
As we are going to work with spheres not necessarily of area 1, say $a$, we shall write $\mu^a_{k, \eta}$ for the quasimorphisms on $\Ham(\sphere^2(a))$.
\begin{rem}
The properties listed in Theorem \ref{prop:c_L} are clearly interrelated, for instance (Shift) is an immediate consequence of (Hofer Lipschitz). The same is true for Theorem \ref{prop:mu_k}, and one remark is needed: for the (Support Control) property we are assuming that the symplectic volume of $\sphere^2$ is 1. Indeed, this properties follows from the (Shift) property of $c_{\underline{L}}$, the definition of $\mu_{k, \eta}$ and its (Lagrangian Control) property. In order to compute $\mu_{k, \eta}(\phi^1_H)$ we indeed need to normalise $H$, and this operation in general consists in defining $H'$ as follows
\[
H'_t(x)=H_t(x)-\frac{1}{Area(\sphere^2)}\int_{\sphere^2}H_t \omega
\]
and (Support Control) reads $\mu^{Area(\sphere^2)}_{k, \eta}(\phi^1_H)=-\frac{1}{Area(\sphere^2)}\text{Cal}(\phi^1_H)$. This fact is not essential but explains the relatively convoluted definition of the quasimorphism appearing in our main result.
\end{rem}

\section{Proof of Theorem \ref{thm:qmorph}}\label{section:proofMain}
The proof of the main result builds upon a generalisation of the (Lagrangian Control) property of the spectral invariants arising from Link Floer theory constructed in \cite{Cghmss21} and described in Section \ref{section:descrTheory}. The strategy goes as follows. Using a continuity property of the action when we look at a diffeomorphism fixing the Lagrangian link, we prove that the spectrum is a point, and then that embedding the disc into spheres of different areas gives a uniform translation of the spectrum. We end the proof finding a relation between the shift in the spectrum and the linking of the braid type defined by the diffeomorphism.

Since the Hamiltonian diffeomorphisms we are lead to work with here are degenerate with respect to the Lagrangian link, we now introduce a simple class of perturbations which we shall use to define the Floer complex.

Fix a perfect Morse function $f$ on $\sphere^1$ of critical values $\pm 1$, and for all $j$ fix a tubular neighbourhood
\begin{equation*}
    \tau_j: \sphere^1\times (-\delta, \delta)\rightarrow\D, \qquad \tau_j(\sphere^1, 0)=L_j
\end{equation*}
of $L_j$ in $\D$. We require the $\tau_j$ to be diffeomorphisms of disjoint images. Let $\rho:(-\delta, \delta)\rightarrow [0, 1]$ be a plateau function, equal to 1 on a neighbourhood of 0 and equal to 0 close to $\pm\delta$. Define
\begin{equation*}
    h_j:\D\rightarrow \R, \qquad h_j(z)=\begin{cases}
        \rho(r)f(\theta) & z=\tau_j(\theta, r)\\
        0 & z\not\in \mathrm{Image}(\tau_j)
    \end{cases}
\end{equation*}
and
\begin{equation*}
   l_\varepsilon:\D\rightarrow \R, \qquad
   l_\varepsilon(z):=\varepsilon\sum_{j=1}^kh_j(z).
\end{equation*}
We can therefore define the Link Floer Homology complex for the Hamiltonian $l_\varepsilon\# H$, which by definition generates the Hamiltonian diffeomorphism $\phi_{l_\varepsilon}\circ \phi_H$. Fix now a point $\textbf{x}$ of the Lagrangian link arbitrarily, and consider $\boldsymbol\alpha_\varepsilon$ reference path for the Floer complex of $l_\varepsilon\# H$ connecting $\phi_{l_\varepsilon}\circ \phi_H(\textbf{x})$ to $\textbf{x}$. The sequence of reference paths $\boldsymbol\alpha_\varepsilon$ converges to $\boldsymbol\alpha$, a reference path for $H$, as $\varepsilon\to 0$. This fact, together with the next lemma, lets us conclude that the spectrum of $l_\varepsilon\# H$ with respect to $\underline L$ converges to the spectrum of $H$ with respect to $\underline L$ as $\varepsilon\to 0$. The latter turns out to be a point, up to periodicity.
\begin{lem}\label{lemma:pointSpectrum}
Assume $H\in\scc^\infty(\sphere^2\times[0, 1]; \R)$ and that $\underline{L}$ is a monotone Lagrangian link on $\sphere^2$. Then if $\phi^1_H(\underline{L})=\underline{L}$ we find $k\mathrm{Spec}(H:\underline{L})=\alpha+\lambda\Z$, where $\alpha\in\R$, $\lambda=2\eta(k_j-1)+A_j$.
\end{lem}
\begin{proof}
By hypothesis $Sym^k(\phi_H)(Sym^k(\underline{L}))=Sym^k(\underline{L})$, and using the homotopy long exact sequence one sees that $\pi_1(X, Sym^k(\underline{L}))=0$. The two conditions taken together imply that all points of the Lagrangian link are homotopic, as constant paths, to the reference path we use in the definition of the action. Since $Sym^k(\underline{L})$ is connected and the spectrum is totally disconnected, if the action is continuous as a function of the points in the Lagrangian link then it is necessarily constant.

As a start, let $y_n\to y$ be a convergent sequence of points in $Sym^k(\underline{L})$, and choose cappings $\hat{y}_n$ for the terms of the sequence. We can choose them in a homotopically consistent way at infinity as follows. Since $y_n\to y$ we have $\phi_H^t(y_n)\to \phi_H^t(y)$ for all $t\in [0, 1]$ by continuity of the flow, the convergence being uniform in $t$ with respect to an auxiliary Riemannian metric on the manifold. Let $d$ be the distance induced by said Riemannian metric. For $N$ big enough and any $n\geq N$, $d(Sym^k(\phi^t_H)(y_n), Sym^k(\phi_H^t)(y))\leq \varepsilon$: we can choose cappings $\hat{y}_n$ such that $\hat{y}_n\sim \hat{y}_m$ for $n, m>N$ for a homotopy supported in an $\varepsilon$ tubular neighbourhood of $y$ as above (essentially, we are considering a local trivialisation of the covering space of capped loops on the loops space, and choosing a fibre).

Under this hypothesis on $\hat{y}_n$, up to changing homotopy class under fixed endpoints of the cappings of the first elements of the sequence, homotopy and reparametrisation, we can assume that for all $m<n$,
\[
\hat{y}_m(s, t)=\hat{y}_n(s, t),\,\, \forall s\in \left[0, 1-\frac{1}{m}\right], \forall t\in [0, 1]
\]
and that
\[
\hat{y}_n(s, t)=y_n(t)\, \, \forall s\in \left[1-\frac{1}{n}, 1\right], \forall t\in [0, 1].
\]
This ensures that the $\hat{y}_n$ converge uniformly to a function $\hat{y}:[0, 1]\times [0, 1]\rightarrow X$, which will be our capping of $y$ of choice. To prove the uniform convergence it suffices to consider the $s$ variable close to 1, but in that case for $m, n$ big enough by triangle inequality $d_\infty(\hat{y}_n, \hat{y}_m)\leq 2\varepsilon$ and the sequence $(\hat{y}_n)_n$ is Cauchy.

With this choice of capping, it is obvious by uniform convergence that $\int_{\D}\hat{y}_n^*\omega_X\rightarrow \int_{\D}\hat{y}^*\omega_X$. We now show that $[\hat{y}_n]\cdot\Delta\rightarrow [\hat{y}]\cdot\Delta$, or equivalently that the intersection number stabilises. Note that $y([0, 1])\cap \Delta=\emptyset$ by uniqueness of solutions of ODEs, so that if we fix a small $\varepsilon$, the $\varepsilon$ tubular neighbourhood of $y$, denoted $V_\varepsilon$, is disjoint from $\Delta$ as well. Let us then choose an $m$ so large that $\forall s\in\left[1-\frac{1}{m}, 1\right]$, $\hat{y}_m(s, \cdot)\subset V_\varepsilon$. For any $n>m$, no new intersections with $\Delta$ can be created in the interval $s\in\left[1-\frac{1}{m}, 1\right]$ since the capping is there contained in $V_\varepsilon$, and all intersections will appear for $s\in \left[0, 1-\frac{1}{m}\right]$. By definition of $\hat{y}$, its intersections with $\Delta$ are located in $s\in \left[0, 1-\frac{1}{m}\right]$ and by definition of the sequence $[\hat{y}]\cdot \Delta=[\hat{y}_m]\cdot \Delta$.

The function $\mathcal{A}^\eta_H: Sym^k(\underline{L})\rightarrow\R/\Z$, $y\mapsto \mathcal{A}^\eta_H(\hat{y})$ for any capping $\hat{y}$ of $y$ is well defined, and the proof above shows that it is continuous. The image of this function is of zero measure by $\sigma$-additivity of the Lebesgue measure, thus nowhere dense and closed. The function $\mathcal{A}^\eta_H$ is therefore necessarily constant by connectedness of $Sym^k(\underline{L})$. Choosing a lift to $\R$ we then find that $k\mathrm{Spec}(H:\underline{L})=\alpha+\lambda\Z$.
\end{proof}

Let us apply Lemma \ref{lemma:pointSpectrum} to a pre-monotone link on $\D^2$ with $k$ components. Let $j_s:\D\rightarrow \sphere^2(1+s)$ be a symplectic embedding into a sphere of area $1+s$, $s\geq 0$, for any $\phi\in\Ham_c(\D)$ denote by $\phi_s \in\Ham(\sphere^2(1+s))$ the obvious extension by the identity, $\underline{L}_s:=j_s(\underline{L})$ the corresponding monotone link on $\sphere^2(1+s)$, and by $\eta_{s}$ the associated monotonicity constant. If $\phi$ is generated by a Hamiltonian $H\in \scc^\infty_c(\D\times \sphere^1; \R)$, let $H_s$ be the Hamiltonian generating $\phi_s$ coinciding with $H$ on the image of $j_s$. Remark now that the periodicity constant of the spectrum does not depend on $s$, it coincides in fact with the area $A$ of a contractible connected component of $\sphere^2(1+s_i)\setminus \bigcup_j L_j$. By Lemma \ref{lemma:pointSpectrum}, $k\mathrm{Spec}(\phi_s, \underline{L}_s)=\alpha(s)+\lambda\Z$. Taking two different values of $s$, $s_1$ and $s_2$,\[
k\mathrm{Spec}(\phi_{s_1}, \underline{L}_{s_1})+K(s_1, s_2)=k\mathrm{Spec}(\phi_{s_2}, \underline{L}_{s_2}).
\]
The constant $K(s_1, s_2)$ is determined up to $\lambda$-periodicity, and a value can be computed choosing any two capped orbits, one for $s_1$ and one for $s_2$. A natural choice is given by a Hamiltonian orbit $y\in Sym^k(\D)$ and cappings in $Sym^k(\sphere^2(1+s_i))$ which are contained, up to homotopy, in $Sym^k(\D)$. Let $\hat{y}_i$, $i=1, 2$, be the two capped orbits, and $\hat{y}$ the preimage of them by $j_{s_i}$. An immediate computation with the definition of the action shows that
\begin{equation}\label{eq:diffAction}\mathcal{A}^{\eta_{s_2}}_{H_{s_2}}(\hat{y}_2)-\mathcal{A}^{\eta_{s_1}}_{H_{s_1}}(\hat{y}_1)=\frac{k}{1+s_2}\Cal(\phi_H)-\frac{k}{1+s_1}\Cal(\phi_H) + (\eta_{s_1}-\eta_{s_2})[\hat{y}]\cdot\Delta.
\end{equation}
The symplectic areas of the cappings in fact coincide in the two cases, the Calabi terms compensate for the different normalisations of the Hamiltonians, while the terms counting the intersections of the cappings with the diagonal only coincide up to the scaling factors $\eta_{s_i}$. The quantity $\eta_{s_1}-\eta_{s_2}$ attains its maximum absolute value when $s_1=(k+1)A-1$ and $s_2=0$, since the monotonicity constant $\eta_s$ has maximum at $s=0$ and minimum at $s=(k+1)A-1$.

\begin{defn}We define
   $Q_k:\Ham_c(\D)\rightarrow \R,$ \begin{equation*}
   Q_k(\phi)=\mu_{k, \eta_{0}}^{1}(\phi_0)+\Cal(\phi)-\left(\mu_{k, \eta_{(k+1)A-1}}^{(k+1)A}(\phi_{(k+1)A-1})+\frac{1}{(k+1)A}\Cal(\phi)\right)
\end{equation*}
\end{defn}

This definition generalises the one given at the beginning of Section \ref{section:B2} to any number of strands.

The next results imply that, if $\phi\in\Ham_{\underline L}(\D, \omega)$, then $-\eta_0[\hat{y}]\cdot \Delta=kQ_k(\phi)$ for some choice of capped intersection point $\hat y$ (the factor $k$ appears in the normalisation one applies when passing from $c_{Sym^k(\underline{L})}$ to $c_{\underline{L}}$, see Section \ref{section:descrTheory}). Fix a diffeomorphism of the sphere
\[
f_{s_1}^{s_2}:\sphere^2(1+s_1)\rightarrow \sphere^2(1+s_2)
\]
and assume for the moment that $s_1, s_2>0$. We require that the following diagram commute:
\[
\begin{tikzcd}
\sphere^2(1+s_1) \arrow[r, "f_{s_1}^{s_2}"]&\sphere^2(1+s_2)\\
\D\arrow[u, "j_{s_1}"]\arrow[ur, "j_{s_2}"]&
\end{tikzcd}.
\]
In particular, since  $j_{s_i}$ is symplectic for $i=1, 2$, by the commutativity of the above diagram we have that $f_{s_1}^{s_2}$ is symplectic between the images of $\D$ in the two different spheres and globally preserves the orientations given by the symplectic forms. $f_{s_1}^{s_2}$ induces a bijective correspondence between the intersection points which, after being capped, generate the Floer complexes $CF(\phi_{s_1}, \underline{L}_{s_1})$ and $CF(\phi_{s_2}, \underline{L}_{s_2})$ for any regular $\phi\in \Ham_c(\D)$. We would like to use $\psi_{s_1}^{s_2}:=Sym^k(f_{s_1}^{s_2})$ to define a chain-isomorphism, and the first step to do so is checking that $\psi_{s_1}^{s_2}$ commutes with the $\Z$-action on both sides given by recapping. Before stating the lemma, recall that $\psi_{s_1}^{s_2}$ is everywhere smooth, because it is a globally defined biholomorphism $\C P^k\rightarrow \C P^k$ in the identification we mentioned in Section \ref{section:descrTheory} (the complex structures on the two copies of $\C P^k$ are then isomorphic, even though one is the standard complex structure and the other is induced by the pushforward by $f_{s_1}^{s_2}$ of the standard complex structure). Two ways to see it for instance are applying the Removable Singularity Theorem on the diagonal (where the function may not be holomorphic), or checking by hand using Cauchy's Integral Formula that, even when a point of $Sym^k (\C P^1)$ has non-trivial isotropy, $\psi_{s_1}^{s_2}$ may be developed in a power series using the coordinates of $\C P^k$.
\begin{lem}
   If $\hat{y}$, $\hat{y}'$ are two equivalent cappings for an intersection point $y$, then $\psi_{s_1}^{s_2}(\hat{y})$, $\psi_{s_1}^{s_2}(\hat{y}')$ are two equivalent cappings for the intersection point $\psi_{s_1}^{s_2}(y)$. Moreover, the associated $\Z$-action induces a translation by multiples of $\lambda$ of the action for both $CF(\phi_{s_1}, \underline{L}_{s_1})$ and $CF(\phi_{s_2}, \underline{L}_{s_2})$.
\end{lem}
\begin{proof}
Start by remarking that $\psi_{s_1}^{s_2}$ induces an identification between the relative homology groups: if $X_{s_i}$ denotes the quotient $\sphere^2(1+s_i)^k/\mathfrak{S}_k$, and $L_{s_i}:=Sym^k(j_{s_1})(Sym^k(\underline{L}))$ $i=1, 2$, then
    \[
    {\psi_{s_1}^{s_2}}_*: H_2(X_{s_1},L_{s_1}; \Z)\xrightarrow{\sim} H_2(X_{s_2}, L_{s_2}; \Z).
    \]
The restriction of ${\psi_{s_1}^{s_2}}_*$ to the image of the Hurewicz morphism $H_2^D(X_{s_1}, L_{s_1}; \Z)\leq H_2(X_{s_1}, L_{s_1}; \Z)$ is still an isomorphism. To check that the equivalence classes of cappings are respected it suffices to check that given two classes $u, u'$ in $H^D_2(X_{s_1}, L_{s_1}; \Z)$ satisfying\[
\langle\omega_{X_{s_1}}, u\rangle+\eta_{s_1}[u]\cdot\Delta_{s_1}=\langle\omega_{X_{s_1}}, u'\rangle+\eta_{s_1}[u']\cdot\Delta_{s_1}
\]
then\[
\langle\omega_{X_{s_2}}, {\psi_{s_1}^{s_2}}_*u\rangle+\eta_{s_2}[{\psi_{s_1}^{s_2}}_*u]\cdot\Delta_{s_2}=\langle\omega_{X_{s_2}}, {\psi_{s_1}^{s_2}}_*u'\rangle+\eta_{s_2}[{\psi_{s_1}^{s_2}}_*u']\cdot\Delta_{s_2}.
\]
Now, Lemma 4.19 from \cite{Cghmss21} shows that\[
\langle\omega_{X_{s_1}}, u\rangle+\eta_{s_1}[u]\cdot\Delta_{s_1}=\frac{\lambda}{2}\mu(u)
\]
for each $u\in H_2^D(X_{s_1}, L_{s_1}; \Z)$ where $\mu\in H^2(X_{s_1}, L_{s_1}; \Z)$ is the Maslov class; an analogous statement holds for classes in $H_2^D(X_{s_2}, L_{s_2}; \Z)$. In light of this, $\hat{y}$ and $\hat{y}'$ are equivalent cappings if and only if their images through the Hurewicz morphism, denoted $u$ and $u'$ respectively, satisfy
\[
\mu(u)=\mu(u').
\]
An analogous result is true for $\psi_{s_1}^{s_2}u$ and $\psi_{s_1}^{s_2}u'$. To prove the first statement it is now enough to show that\[
\mu(u)=\mu({\psi_{s_1}^{s_2}}_*u)
\]
and the analogous conclusion for $u'$. These are true because $\psi_{s_1}^{s_2}$ is a biholomorphism, and thus preserves the Maslov indices (see \cite[Theorem C.3.7]{McDuffSalamon94}).

To prove the second statement, we need to prove that the generators of $H_2^D(X_{s_i}, L_{s_i}; \Z)$ act the same way on the action for $i=1, 2$. Notice that in both cases $\lambda_1=\lambda_2=A=:\lambda$ the area of a disc bounded by a component of the Lagrangian link. Now, capping by one of the generators of $H_2^D(X_{s_i}, L_{s_i}; \Z)$ shifts the action by $\lambda=A$ in both cases since the generators have Maslov index 2, see Corollary 4.8 and Lemma 4.19 in \cite{Cghmss21}.
\end{proof}
Fix an $\omega$-tame almost complex structure on $\sphere^2(1+s_1)$ and push it forward by $\psi_{s_1}^{s_2}$. Let $H$ be a Hamiltonian generating $\phi$, and $\Psi_{s_1}^{s_2}$ be the morphism of Floer groups induced by $\psi_{s_1}^{s_2}$ defined the following way. If $\hat y$ is a capped intersection point in $L_{s_1}\pitchfork Sym^k(\phi_{s_1})L_{s_1}$,
\begin{equation*}
    [\hat y]\mapsto [\psi_{s_1}^{s_2}\circ \hat y].
\end{equation*}
Remark that $\psi_{s_1}^{s_2}\circ \hat y$ is a capped intersection point in $L_{s_2}\pitchfork Sym^k(\phi_{s_2})L_{s_2}$. Extend the above definition by $\Lambda$-linearity to obtain the morphism of Floer chain groups $\Psi_{s_1}^{s_2}: CF(\phi_{s_1}, \underline{L}_{s_1})\rightarrow CF(\phi_{s_2}, \underline{L}_{s_2})$.

\begin{lem}\label{lemma:chainMap}
Let $\phi:\D\rightarrow \D$ be a Hamiltonian diffeomorphism of the disc, such that $\phi(\underline{L})\pitchfork\underline{L}$. Then $\Psi_{s_1}^{s_2}$ is a chain isomorphism, and in particular $HF(\phi_{s_1}, \underline{L}_{s_1})\simeq HF(\phi_{s_2}, \underline{L}_{s_2})$.
\end{lem}
\begin{proof}
Let us denote by $\partial^{s_i}$, $i=1, 2$, the differentials of the two complexes. We want to prove that\[
\partial^{s_2}\circ \Psi_{s_1}^{s_2}=\Psi_{s_1}^{s_2}\circ \partial^{s_1}.
\]
To define the Floer complex, one needs to achieve transversality in the symmetric product. We consider a class of almost complex structures which coincide with the one induced by the projection to the quotient
\begin{equation}
    \label{eq:quotProj}\sphere^2(1+s_i)^k\rightarrow\sphere^2(1+s_i)^k/\mathfrak{S}_k=:X_{s_i}
\end{equation}
near the diagonal. For $i=1, 2$ let
\[
\mathcal{J}_i(\Delta)
\]
be the set of almost complex structures on $Sym^k(\sphere^2(1+s_i))$ which coincide with $J^{s_i}_{X_{s_i}}$ on a neighbourhood of the fat diagonal $\Delta\subset X_{s_i}$, and elsewhere they are tamed by $\omega_{X_{s_i}}$. Here we use $\omega_{X_{s_i}}$ and $J^{s_i}_{X_{s_i}}$ for the natural (singular) symplectic form and almost complex structure on $X_{s_i}$ induced by the projection to the quotient (\ref{eq:quotProj}).

Let now $u:\R\times \sphere^1\rightarrow \Sym^k(\sphere^2(1+s_1))$ be a smooth function satisfying the conditions
\begin{equation}\label{eq:holCond}
\begin{cases}
u(s, 0)\in\Sym^k(\phi_{s_1})(Sym^k(\underline{L}))\\
u(s, 1)\in\Sym^k(j_{s_1})(Sym^k(\underline{L}))\\
\lim_{s\to -\infty}u(s, t)=j_{s_1}y_0\\
\lim_{s\to +\infty}u(s, t)=j_{s_1}y_1\\
\partial_s u(s, t)+J^{s_1}_t\partial_t u(s, t)=0
\end{cases}.\end{equation}
where $J^{s_1}\in\mathcal{J}_1(\Delta)$ is an almost complex structure on $Sym^k(\sphere^2(1+s_1))$. Endow $Sym^k(\sphere^2(1+s_2))$ with the almost complex structure $J^{s_2}$ defined as follows. On a small neighbourhood of $\Delta$ we set $J^{s_2}:=J^{s_2}_{X_{s_2}}$, and elsewhere we put $J^{s_2}:={\psi_{s_1}^{s_2}}_*J^{s_1}$ (the push-forward by $\psi_{s_1}^{s_2}$ of $J^{s_1}$), so that by definition the differential of $\psi_{s_1}^{s_2}$ is complex linear.

We want then to show that $J^{s_2}$ thus defined is an element of $\mathcal{J}_2(\Delta)$: we only need to show that $J^{s_2}$ is tamed by $\omega_{X_{s_2}}$ if $J^{s_1}$ is tamed by $\omega_{X_{s_1}}$. Tameness being an open condition, and since the natural complex structure is tamed by the natural symplectic structure, let us assume for the moment that $J^{s_1}$ is close enough to $J^{s_1}_{X_{s_1}}$. The pushforward ${\psi_{s_1}^{s_2}}_*$ defines a homeomorphism between the spaces of almost complex structures on $X_{s_1}$  and $X_{s_2}$, and sends by definition $J^{s_1}_{X_{s_1}}$ to $J^{s_2}_{X_{s_2}}$. This implies that if $J^{s_1}$ is close enough to $J^{s_1}_{X_{s_1}}$ then $J^{s_2}:=(\psi_{s_1}^{s_2})_*J^{s_1}$ is close to $J^{s_2}_{X_{s_2}}$ and is tame: $J^{s_2}\in\mathcal{J}_2(\Delta)$.

A bijective correspondence between holomorphic curves $u_1: \D\rightarrow X_{s_1}$ and $u_2: \D\rightarrow X_{s_2}$ is given by the post-composition with $\psi_{s_1}^{s_2}$ (or its inverse). We are left to show that the differentials are defined at the same time, i.e. that if $J^{s_1}$ achieves the transversality required in the definition of the
Floer differential, then so does $J^{s_2}$. As mentioned above we have a well defined homeomorphism
\[
(\psi_{s_1}^{s_2})_*:\mathcal{U}^1\xrightarrow {\sim}\mathcal{U}^2
\]
where $\mathcal{U}^i\subset\mathcal{J}_i(\Delta)$ are neighbourhoods of the natural almost complex structures; the inverse is given by the pullback via $\psi_{s_1}^{s_2}$. If $\mathcal{J}_{i, \tau}(\Delta)\subset \mathcal{J}_i(\Delta)$ is the Baire set of almost complex structures giving good definition for the Floer differential (proof in \cite[Section 5]{Cghmss21}), let us prove that $\mathcal{J}_{1, \tau}(\Delta)\cap (\psi_{s_1}^{s_2})^*(\mathcal{U}^2\cap\mathcal{J}_{2, \tau}(\Delta))\neq\emptyset$. If it is the case, we may choose $J^{s_1}$ in it, and push it forward to $J^{s_2}$, so both Floer complexes can be defined with these choices. 
Now, the intersection cannot be empty: since $\mathcal{J}_{2,\tau}(\Delta)$ is generic in $\mathcal{J}_2(\Delta)$, $\mathcal{U}^2\cap\mathcal{J}_{2,\tau}(\Delta)$ is a Baire set in $\mathcal{U}^2$, and $(\psi_{s_1}^{s_2})^*(\mathcal{U}^2\cap\mathcal{J}_{2, \tau}(\Delta))$ is Baire in $\mathcal{U}^1$. Intersecting the latter with $\mathcal{J}_{1, \tau}(\Delta)$ gives a Baire set in $\mathcal{U}^1$, which is in particular not empty.

As for the orientation of moduli spaces, one can arbitrarily define a spin structure on $Sym^k(\underline{L})$ already on the disc, that which gives spin structures on $Sym^k(\underline{L}_{s_2})$ and $Sym^k(\underline{L}_{s_1})$ by pushforward by $j_{s_i}$, and then these two correspond by pushforward by $\psi_{s_1}^{s_2}$ since $\psi_{s_1}^{s_2}$ preserves the orientation of the Lagrangian link.
\end{proof}

\begin{rem}
    The isomorphism between the Floer complexes in Lemma \ref{lemma:chainMap} is an isomorphism of persistence modules only up to shift. What we are going to do later essentially amounts to computing how much it fails to preserve the action filtration. 
\end{rem}
Let $Q\Psi_{s_1}^{s_2}:QH(Sym^k(\underline{L}_{s_1}); \Z)\rightarrow QH(Sym^k(\underline{L}_{s_2}); \Z)$ be the morphism on quantum homology induced by $\psi_{s_1}^{s_2}$.
\begin{lem}\label{lemma:PSS}
Let $PSS(s_i): QH(Sym^k(\underline{L}_{s_i}); \Z)\rightarrow HF(\phi_{s_i}, \underline{L}_{s_i}; \Z)$ be the PSS isomorphisms for $i=1, 2$. Then the following diagram commutes:
\[
\begin{tikzcd}
QH(Sym^k(\underline{L}_{s_1}); \Z)\arrow[d, "PSS(s_1)"] \arrow[r, "Q\Psi_{s_1}^{s_2}"]& QH(Sym^k(\underline{L}_{s_2}); \Z)\arrow[d, "PSS(s_2)"]\\
HF(\phi_{s_1}, \underline{L}_{s_1};  \Z)\arrow[r, "\Psi_{s_1}^{s_2}"]&HF(\phi_{s_2}, \underline{L}_{s_2};  \Z)
\end{tikzcd}.
\]
\end{lem}
\begin{proof}
The definition of the PSS isomorphism and its bijectivity in the case of symmetric products are addressed in \cite[Section 6.2]{Cghmss21}. Modifying the (singular) Hamiltonian $Sym^k(H)$ near the diagonal, one can find  an
isomorphic filtered chain complex for which the standard proofs work. In the following we assume that this operation has been done. In particular, as per \cite[Lemma 6.11]{Cghmss21}, pearly  models are used in the definitions of the maps $PSS(s_i)$: in the following we discuss how we can use $\psi_{s_1}^{s_2}$ to push forward the pearly model defining $PSS(s_1)$ to a pearly model defining $PSS(s_2)$.

As in the proof for $\Psi_{s_1}^{s_2}$, we need to check that $\psi_{s_1}^{s_2}$ induces a bijective correspondence of generators of the chain complexes and of moduli spaces defining the differential. This is immediately clear for $Q\Psi_{s_1}^{s_2}$: gradient lines on $Sym^k(\underline{L}_{s_1})$ are mapped to gradient lines if one fixes a Morse-Smale pair and then pushes it forward, and to holomorphic discs in the pearly model for $QC(Sym^k(\underline{L}_{s_1}))$ we make bijectively correspond holomorphic discs in the pearly model for $QC(Sym^k(\underline{L}_{s_2}))$. As for the orientation of the involved moduli spaces, we argue as at the end of Lemma \ref{lemma:chainMap}: the spin structures on the Lagrangian links are sent one to the other by $\psi_{s_1}^{s_2}$. This ends the proof that $Q\Psi_{s_1}^{s_2}$ is an isomorphism. For the commutativity of the diagram, one just needs to choose the PSS-data for $PSS(s_1)$ and push it forward by $\psi_{s_1}^{s_2}$ as done in Lemma \ref{lemma:chainMap}. In particular, one needs to make sure that achieving simultaneous transversality for all the involved definitions is possible, and this is done precisely as in Lemma \ref{lemma:chainMap}. The moduli spaces involved in the definition of $PSS(s_i)$, $i=1, 2$, are then zero-dimensional smooth compact manifold as desired, and they can be shown to be in bijection using a post-composition by $\psi_{s_1}^{s_2}$ as we did in the proof of Lemma \ref{lemma:chainMap}.
\end{proof}

The next lemma is crucial to determine the relation between the spectral invariants associated to embeddings into spheres of different areas.
\begin{lem}\label{lemma:sameLinComb}
    If $\hat{x}, \hat{y}$ are intersection points for $(Sym^k(H_{s_1}), Sym^k(\underline{L}_{s_1}))$ with cappings, then\[
    \mathcal{A}^{\eta_1}_{H_{s_1}}(\hat{x})-\mathcal{A}^{\eta_2}_{H_{s_2}}(\Psi_{s_1}^{s_2}\hat{x})=\mathcal{A}^{\eta_1}_{H_{s_1}}(\hat{y})-\mathcal{A}^{\eta_2}_{H_{s_2}}(\Psi_{s_1}^{s_2}\hat{y}).
    \]
\end{lem}
\begin{proof}
First observe that since the monotonicity constant $\lambda$ does not depend on the embedding we choose, we may recap the intersection points in such a way that the capping is contained in $Sym^k(\D)$ without changing either side of the equation. In this setting, by Equation \ref{eq:diffAction}, we are left to prove that
\[
(\eta_{s_2}-\eta_{s_1})[\hat{x}]\cdot\Delta=(\eta_{s_2}-\eta_{s_1})[\hat{y}]\cdot\Delta.
\]
If $\hat{x}$ could be homotoped to $\hat{y}$ without adding intersections with the diagonal, using a path in $Sym^k(\underline{L})$, this would be true. To be able to suppose this, we remark that we may recap either intersection point in such a way that both cappings are still contained in $Sym^k(\D)$ and are moreover homotopic, using the action of $\pi_2(Sym^k(\D), Sym^k(\underline{L}))$. This does not affect intersection products with the diagonal, which is what we are interested in, because we now recap using relative homology classes $u_i$ such that $u_i\cdot\Delta=0$.
\end{proof}
Let us shorten the notation, defining
\begin{equation*}
    c'_{\underline{L}_s}(H_s):=c_{\underline{L}_s}(H_s)+\frac{1}{1+s}\Cal(\phi_H).
\end{equation*}
The next lemma will prove the main result.
\begin{lem} If $\phi_H\in\Ham_{\underline{L}}(\D, \omega)$,
\begin{equation*}
  c'_{\underline{L}_0}(H_{0})- c'_{ \underline{L}_{(k+1)A-1}}(H_{(k+1)A-1})=\frac{1}{k}\eta_0 \mathrm{lk}(b(\phi_H, \underline{L})).  
\end{equation*}
\end{lem}
\begin{proof}
By the definition of $\Ham_{\underline{L}}(\D, \omega)$, $\phi_H$ cannot be non-degenerate, since the intersections between the Lagrangian link and its deformation under $\phi_H$ are not transversal. For the moment we consider small perturbations of $\phi_H$ generated by $l_\varepsilon\# H$ (as at the beginning of this Section). Let $y\in Sym^k(\underline{L})\cap Sym^k\phi_{l_\varepsilon\# H}(Sym^k(\underline{L}))$, and $\hat{y}_i$ a capping in $X_{s_i}$, $i=1, 2$. As we saw in the proof of Lemma \ref{lemma:sameLinComb}, by Equation \ref{eq:diffAction}\begin{align*}
    &\mathcal{A}^{\eta_{s_2}}_{(l_\varepsilon\# H)_{s_2}}(\hat{y}_2)-\mathcal{A}^{\eta_{s_1}}_{(l_\varepsilon\# H)_{s_1}}(\hat{y}_1)=\\=\frac{k}{1+s_2}\Cal & (\phi_{l_\varepsilon\# H})-\frac{k}{1+s_1}\Cal(\phi_{l_\varepsilon\# H})+(\eta_{s_1}-\eta_{s_2})[\hat{y}]\cdot\Delta
\end{align*}
for all $s_1, s_2\in \left(0, (k+1)A-1\right)$ if all the cappings are included in $Sym^k(\D)$. By Lemma \ref{lemma:sameLinComb} we know that the action of capped intersection points is uniformly shifted when the area of the sphere changes, by Lemmata \ref{lemma:chainMap} and \ref{lemma:PSS} we know that a chain in $CF(\phi_{(l_\varepsilon\# H)s_1}, \underline{L}_{s_1})$ represents the fundamental class if and only if its image by ${\Psi_{s_1}^{s_2}}$ does: $c_{\underline{L}_{s_1}}((l_\varepsilon\# H)_{s_1})$ and $c_{\underline{L}_{s_2}}((l_\varepsilon\# H)_{s_2})$ are attained by a linear combination of capped intersection points which are identified by $\Psi_{s_1}^{s_1}$.
Therefore, since the right hand side of Equation \ref{eq:diffAction} is continuous as a function of the parameters $s_1, s_2$, we find
\[
c_{\underline{L}_{0}}((l_\varepsilon\# H)_0)-c_{\underline{L}_{(k+1)A-1}}((l_\varepsilon\# H)_{(k+1)A-1})=\frac{1}{k}\left(\mathcal{A}_{0}([\hat{y}])-\mathcal{A}_{(k+1)A-1}([\hat{y}])\right)
\]
and recapping the orbit $y$ does not change this difference. Without loss of generality we may assume that all cappings are contained in $Sym^k(\D)$ where $\psi_{s_1}^{s_2}$ is symplectic. The symplectic area contributions in the action vanish and we are left with
\begin{equation}\label{eq:diffSpecInvInter}
    c'_{\underline{L}_{0}}((l_\varepsilon\# H)_0)-c'_{\underline{L}_{(k+1)A-1}}((l_\varepsilon\# H)_{(k+1)A-1})=\frac{\eta_{0}}{k}[\hat{y}]\cdot \Delta
\end{equation}

where $[\hat{y}]$ is a homotopy between a constant braid and the braid $b(\phi, \underline{L})$ which is fully contained in $Sym^k(\D)$. We now claim that using the (Hofer Lipschitz) and the (Spectrality) axioms, together with the Hofer Lipschitz property of the Calabi homomorphism, we can show that this equality is still true in the degenerate case, i.e. if $\phi\in\Ham_{\underline{L}}(\D, \omega)$. It suffices to say that the right-hand side is constant under $\scc^2$-small perturbations of the Hamiltonian, which means that we have to show that, taking a sequence of cappings for $l_\varepsilon\# H$ as $\varepsilon\to 0$, the intersection number with the diagonal does not change for $\varepsilon$ small enough. As done at the beginning of Section \ref{section:proofMain}, we assume that the reference path for $l_\varepsilon\# H$ is constructed on a point $\textbf{x}$ which is independent from $\varepsilon$. Now, to compute the difference in (\ref{eq:diffSpecInvInter}) we may use an intersection point in $Sym^k(\underline{L})\cap Sym^k(\phi_{l_\varepsilon\# H})Sym^k(\underline{L})$ which does not depend on $\varepsilon$. The set $Sym^k(\underline{L})\cap Sym^k(\phi_{l_\varepsilon\# H})Sym^k(\underline{L})$ does in fact not depend on $\varepsilon$ (as long as it is small and positive), and by Lemma \ref{lemma:sameLinComb} the action is uniformly shifted by the chain-complex isomorphism\begin{equation*}
    CF((\phi_{l_\varepsilon\# H})_{s_1}, \underline{L}_{s_1})\rightarrow CF((\phi_{l_\varepsilon\# H})_{s_2}, \underline{L}_{s_2})
\end{equation*}for all $\varepsilon$ and $s$. Therefore, to compute the left-hand side of (\ref{eq:diffSpecInvInter}) we may choose an intersection point in $Sym^k(\underline{L})\cap Sym^k(\phi_{l_\varepsilon\# H})Sym^k(\underline{L})$ which does not depend on $\varepsilon$ as a point in $Sym^k(\D)$, together with cappings in $Sym^k(\D)$ which do depend on $\varepsilon$, since so does the reference path. However, since we have $\mathscr{C}^\infty$-convergence of the reference paths for the Floer complexes of $(\phi_{l_\varepsilon\# H})_s$ as $\varepsilon\to 0$ to a reference path for $(\phi_H)_s$, to compute the left-hand side of (\ref{eq:diffSpecInvInter}) we are allowed to take cappings which $\mathscr{C}^\infty$-converge as $\varepsilon$ tends to 0. This proves that the quantity $[\hat y]\cdot \Delta$ is constant for small $\varepsilon$, which is what we wanted to show, and the equality in (\ref{eq:diffSpecInvInter}) holds for diffeomorphisms in $\Ham_{\underline{L}}(\D, \omega)$.

To conclude now we need to prove that \begin{equation}\label{eq:int=link}
   [\hat{y}]\cdot \Delta=\mathrm{lk}(b(\phi_H, \underline{L})) 
\end{equation}
where $y$ is a point in $Sym^k(\underline{L})$ and $\hat y$ is a capping for $\phi_H$ based at $y$ whose image is entirely contained in $Sym^k(\D)$.
Without loss of generality we may assume that the braid $b(\phi_H, \underline{L})$ is pure, as follows. First off, since $\phi_H$ belongs in $\Ham_{\underline{L}}(\D, \omega)$, so does $\phi_H^n$ for all $n\in \N$. Note that both sides of Equation \ref{eq:int=link} are homogeneous under iteration of $\phi_H$. Indeed, for the right hand side we in fact already know that the linking number is a group homomorphism, the discussion for the left hand side being slightly more complicated we are going to prove it by induction. It suffices to check the statement for $\phi_H\circ\phi_H$ (the case $\phi^n_H\circ \phi_H$ being analogous), i.e. given a capped intersection point $\hat{y}$ for $\phi_H$ whose image lies inside $Sym^k(\D)$, we are going to construct a capped intersection point for $\phi^2_H$ whose image is still contained, up to homotopy, in $Sym^k(\D)$, and which is homotopic to a concatenation of cappings. Recall that $\hat{y}:[0, 1]\times [0, 1]\rightarrow Sym^k(\D)$ satisfies $\hat{y}(0, t)=y$, $\hat{y}(1, t)=\boldsymbol\alpha_1(t)$ for some reference Hamiltonian path $\boldsymbol\alpha_1$. Let $\boldsymbol{\alpha}_2$ be a reference path for $\phi\circ\phi$. We remark that modifying $\boldsymbol{\alpha}_1$ in such a way that $\boldsymbol{\alpha}_1(0)=\boldsymbol{\alpha}_1(1)$ using a path contained in $Sym^k(\underline{L})$ we may concatenate $\boldsymbol\alpha_1$ with itself. Adding to this concatenation another path in $Sym^k(\underline{L})$ and denoting the resulting path $\tilde{\boldsymbol{\alpha}}$,
\[
\tilde{\boldsymbol{\alpha}}(i)=\boldsymbol{\alpha}_2(i)\,\, \mathrm{for}\, i=0, 1.
\]
The two paths are moreover homotopic with fixed endpoints in $\mathrm{Conf}^k(\D)$. We are going to construct a homotopy in $Sym^k(\D)$ between $y$ and $\tilde{\boldsymbol{\alpha}}$, $\tilde{y}$, in such a way that it is clear that $[\tilde{y}]\cdot\Delta=2[\hat{y}]\cdot\Delta$. We define
\[
\tilde{y}=\hat{y}\#_t h_1\#_t \hat{y}\#_th_2
\]
where if $u, v:[0, 1]\times[0, 1]\rightarrow Sym^k(\D)$
\[
u\#_tv(s, t)=\begin{cases}
    u(s, 2t)\,\, \mathrm{for}\,\, t\in\left[0, \frac{1}{2}\right]\\
    v(s, 2t-1)\,\, \mathrm{for}\,\, t\in\left[\frac{1}{2}, 1\right]
\end{cases}
\]
is the concatenation along the $t$ variable, and $h_1$, $h_2$ are homotopies to the paths we had to interpose to respectively concatenate $\hat{y}$ with itself and to join the endpoints of $\boldsymbol{\alpha}_1$ with those of $\boldsymbol{\alpha}_2$.
By construction, in $\tilde{y}$ we may have intersections with the diagonal only in correspondence of the components $\hat{y}$, and each occurs twice with the same sign. This proof generalises the same way to arbitrary compositions of diffeomorphisms.

This proves that, if we compose an arbitrary number of copies of $\phi_H$ and divide both sides of Equation \ref{eq:int=link} by the same number, we may suppose without loss of generality that for all $i$, $\phi_H(L_i)=L_i$. Remember that restricting to homotopies contained in $\D$ the number $[\hat{y}]\cdot\Delta$ does not depend on the homotopy we use to compute it. We apply now the study of intersections in the symmetric product as developed in \cite[Section 4.2]{morTri24}: there the authors describe a class of intersections in the symmetric product in terms of braid type which are guaranteed to be transverse, called ``elementary intersections''. We refer to that article for details.

Consider a homotopy between $y$ and a trivial braid that starts by unlinking the $k$-th strand (we are here reversing the $s$-variable of the capping for clarity in the explanation, but we need to keep track of this in the sign of the intersections). Every time the $k$-th strand is involved in a crossing and passes under the other strand, introduce a pair of negative elementary intersections with the diagonal that makes the $k$-th  strand pass over the other. Every time we apply this operation, the linking number decreases by $2$ (see Figure \ref{figure:elIntersection}). After these modifications, the braid is reduced to one in $\mathcal{B}_{k-1}$ together with a constant strand, and the difference in linking number with the original braid coincides with the algebraic number of elementary intersections we introduced.

 \begin{figure}
	\centering
\begingroup%
  \makeatletter%
  \providecommand\color[2][]{%
    \errmessage{(Inkscape) Color is used for the text in Inkscape, but the package 'color.sty' is not loaded}%
    \renewcommand\color[2][]{}%
  }%
  \providecommand\transparent[1]{%
    \errmessage{(Inkscape) Transparency is used (non-zero) for the text in Inkscape, but the package 'transparent.sty' is not loaded}%
    \renewcommand\transparent[1]{}%
  }%
  \providecommand\rotatebox[2]{#2}%
  \newcommand*\fsize{\dimexpr\f@size pt\relax}%
  \newcommand*\lineheight[1]{\fontsize{\fsize}{#1\fsize}\selectfont}%
  \ifx\svgwidth\undefined%
    \setlength{\unitlength}{322.0388842bp}%
    \ifx\svgscale\undefined%
      \relax%
    \else%
      \setlength{\unitlength}{\unitlength * \real{\svgscale}}%
    \fi%
  \else%
    \setlength{\unitlength}{\svgwidth}%
  \fi%
  \global\let\svgwidth\undefined%
  \global\let\svgscale\undefined%
  \makeatother%
  \begin{picture}(1,0.42809446)%
    \lineheight{1}%
    \setlength\tabcolsep{0pt}%
    \put(0,0){\includegraphics[width=\unitlength,page=1]{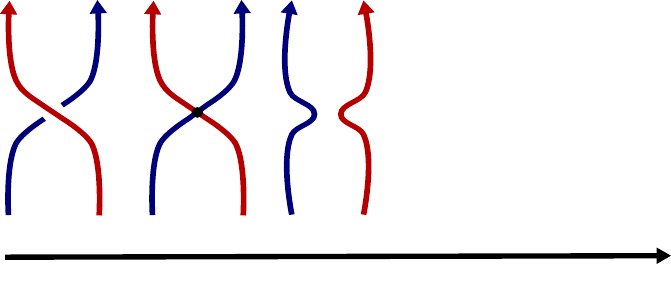}}%
    \put(0.92417473,0.00507598){\color[rgb]{0,0,0}\makebox(0,0)[lt]{\lineheight{1.25}\smash{\begin{tabular}[t]{l}$s$\end{tabular}}}}%
    \put(0,0){\includegraphics[width=\unitlength,page=2]{elIntersection.pdf}}%
  \end{picture}%
\endgroup%

	\caption{We represent in red the $k$-strand of the braid. The time $s$ is the time of the capping: at $s=0$ we have the intersection point (constant path), while at $s=1$ we have the reference path.  Here we represent $s$ only for a small segment in $(0, 1)$, and we focus on two strands only. The represented elementary intersection is positive. The sign difference from the text is due to the fact that we prefer to describe the unlinking process, for which we read this figure from right to left, effectively reversing the $s$-time and the signs of intersections.}
	\label{figure:elIntersection}
    \end{figure}

Proceed now inductively on the other strands. Since capped intersection points are in fact homotopies between the trivial braid and a braid made of Hamiltonian paths from $Sym^k(\phi_H)(Sym^k(\underline{L}))$ to $Sym^k(\underline{L})$, the number we obtain this way is $\mathrm{lk}(b(\phi_H, \underline{L}))$.
\end{proof}
This lemma promptly yields the main result: for any $\phi_H\in \Ham_{\underline{L}}(\D, \omega)$, for any $n\in \N$, $\phi_H^n\in \Ham_{\underline{L}}(\D,\omega)$ with $b(\phi_H^n, \underline{L})=b(\phi_H, \underline{L})^{\# n}$. Then
\[
\frac{1}{n}\left[c'_{ \underline{L}_{0}}(H_0^{\# n})-c'_{\underline{L}_{(k+1)A-1}}(H_{(k+1)A-1}^{\# n})\right]=\frac{1}{n}\frac{1}{k}\eta_0 n\cdot \mathrm{lk}(b(\phi_H, \underline{L}))=\frac{1}{k}\eta_0 \mathrm{lk}(b(\phi_H, \underline{L})).
\]
To find the theorem, one needs to note that\[
\eta_0=\frac{(k+1)A-1}{2(k-1)}
\]
as it follows from the equation
\[
A=\lambda=2\eta_0(k-1)+1-(k+1)A
\]
and to take the limit for $n\to\infty$.
\begin{rem}
From the proof we see that the main result is still true for the spectral invariants before homogenisation.
\end{rem}

\section{The case of $\mathcal{B}_2$}\label{section:B2}
In this section we are going to present a different proof of Theorem \ref{thm:qmorph} for $k=2$. The methods here generalise verbatim to higher numbers of strands, but to diffeomorphisms of certain braid types only.

M. Khanevsky in \cite{kha11} proved that given a non-displaceable disc contained in an annulus, the Hofer norm of a diffeomorphism fixing such a disc (not necessarily pointwise) is in a linear relation with the rotation number of the disc. As we see in this section, Khanevsky's proof may be generalised in an elementary way to our setting. We are going to consider two disjoint discs in $\D$ whose boundaries constitute a premonotone Lagrangian configuration $\underline{L}$, and give lower bounds on the Hofer norm of an element in $\Ham_{\underline{L}}(\D, \omega)$ using the linking number of the induced braid.

Let $\D\subset\C$ be the two dimensional unit open disc, with area form $\omega$ normalised so that $\int_\D\omega=1$. Let $D_1$, $D_2$ be two open disjoint discs in $\D$ of same area such that neither is displaceable in the complement of the other (then $A=Area(D_i)\in \left(\frac{1}{3}, \frac{1}{2}\right)$). Let $L_i:=\partial D_i$, and $\underline{L}=L_1\times L_2$, and assume moreover that $\phi(D_i)=D_i$. Remember that for a compactly supported Hamiltonian diffeomorphism $\phi\in \Ham_{\underline{L}}(\D, \omega)$ we can define its braid type, let us call it $b(\phi, \underline{L})\in\mathcal{B}_2$. Since $\mathcal{B}_2$ is isomorphic to $\Z$ via the linking number, we identify $b(\phi, \underline{L})$ with $\frac{1}{2}\cdot \mathrm{lk}(b(\phi, \underline{L}))$ (the factor of 2 is needed since in our definition of linking number agrees with the group-theoretic one, and the elementary loop $t\mapsto \exp(2\pi i t)$ has linking 2 with the origin of the complex plane). Since $\phi(D_i)=D_i$, $\frac{1}{2}\cdot \mathrm{lk}(b(\phi, \underline{L}))\in \Z$.

In a similar way as one can find in \cite{kha11}, we define the set
\[
S_n=\Set{\phi\in \Ham_{\underline{L}}(\D) | \phi(D_i)=D_i,\, \frac{1}{2}\mathrm{lk}(b(\phi, \underline{L}))=n}.
\]
If $\hat{\Phi}$ is a Hamiltonian diffeomorphism such that $\hat{\Phi}\in S_1$, we have a decomposition $S_n=\hat{\Phi}^nS_0$. We want to prove that if $\phi\in S_n$ then its Hofer norm satisfies
\[
\norm{\phi}\geq O(n)
\]
and to do this, we show that there is a quasimorphism which is Hofer-Lipschitz, strictly positive on $\hat{\Phi}$ and zero on $S_0$.

Consider polar coordinates on the disc, so that the symplectic form reads $\omega=\frac{1}{\pi}r\,dr\wedge d\theta$. Up to conjugation by a symplectic diffeomorphism (which does not change the computations of Hofer norms) we can suppose that the picture is symmetric, the two discs $D_i$ being found in the halfplanes of positive and negative $x$. Let $\hat{\Phi}$ be a compactly supported smooth approximation of a rotation of $2\pi$, a generating Hamiltonian being for instance \[
H(r, \theta)=\rho(r^2)r^2
\]
where $\rho$ is a plateau function which is 1 outside a small neighbourhood of 1 ($\rho$ will in general depend on the area of $D_i$, and we choose it in a way that $(r, \theta)\mapsto\rho(r^2)$ is equal to 1 in a neighbourhood of $D_1\cup D_2$). Let $j_s$ be a symplectic embedding from $\D$ into a sphere of area $1+s$, which we denote with $\sphere^2(1+s)$. It induces a Hofer 1-Lipschitz injection $j_{s\,*}:\Ham_c(\D)\rightarrow \Ham(\sphere(1+s))$ given by extension by 0 of the Hamiltonians. Observe that the configuration $\underline{L}_s:=j_s(\underline{L})$ is monotone in $\sphere^2$; let $\eta_s$ be its monotonicity constant that appears as an $\eta$ in Equation \ref{eq:monotonicity}.

The quasimorphism on $\Ham_c(\D)$ of our choice is \[
Q_2:=\mu^1_{2, \eta_0}\circ j_{0\, *}+\Cal-\left(\mu^{3A}_{2, \eta_{3A-1}}\circ j_{(3A-1)\, *}+\frac{1}{3A}\Cal\right).
\]

Its Hofer Lipschitz constant is $3+\frac{1}{3A}$ and we may see as follows that they are furthermore $\scc^0$-continuous by the criterion by Entov-Polterovich-Py (Theorem 3 in \cite{EPP08}, see also \cite{Cghmss21}). If we show that $Q_2$ vanishes on any Hamiltonian diffeomorphism supported on a disc of area less than $A$, an application of the criterion gives us $\scc^0$-continuity of $Q_2$. Using the first property in Theorem \ref{prop:mu_k}, $\mu^a_{k, \eta}$ only depends on $a, k$ and $\eta$. Given $\phi$ supported on a disc of area less than $A$, in order to compute $\mu^1_{2, \eta_0}\circ j_{0\, *}(\phi)$  and $\mu^{3A}_{2, \eta_{3A-1}}\circ j_{(3A-1)\, *}(\phi)$ we choose then circles contained in the complement of $\mathrm{Supp}(\phi)$ in the spheres. An application of the (Support Control) property of $\mu^a_{k, \eta}$ shows that $Q_2(\phi)=0$ and therefore the $\scc^0$-continuity of $Q_2$.

Write $Q_2=\mu'_0-\mu'_{3A-1}$, where
\[
\mu'_s=\mu^{1+s}_{2, \eta_{s}}\circ j_{s\, *}+\frac{1}{1+s}\Cal.
\]

Let us compute $Q_2(\hat\Phi)$: by the first property in Theorem \ref{prop:mu_k} for each of the $\mu'_s$ we can use a configuration of two circles, both contained in $j_s(\D)$ and whose centres coincide with $j_s(0)$. The area requirement translates to the condition that the radii associated to the two circles are $r=\sqrt{A}$ and $r=\sqrt{1+s-A}$. Then by (Lagrangian Control), we find
\[
\mu'_s(\hat{\Phi})=\frac{s+1}{2}
\]
so that
\[
Q_2(\hat{\Phi})=\frac{1}{2}-\frac{1+3A-1}{2}=\frac{1-3A}{2}.
\]

Following the steps of \cite{kha11}, and assuming that $Q_2$ does vanish on $S_0$ (we are going to later show that this is indeed the case), let us prove that if $\phi=\hat{\Phi}^n\psi$ for a diffeomorphism $\psi\in S_0$, then $Q_2(\phi)=Q_2(\hat{\Phi}^n)=nQ(\hat{\Phi})$. Given such a decomposition of $\phi$, one has $Q_2(\phi)=Q_2(\hat{\Phi}^n)+D(\phi)$, where $D(\phi)\in\R$ is bounded in modulus by the defect of the quasimorphism $Q_2$ (in particular, it is bounded uniformly on $\phi$). Since $Q_2$ is homogeneous, for all positive integers $k$
\[
Q_2(\phi)=\frac{Q_2(\phi^k)}{k}=\frac{Q_2(\hat{\Phi}^{nk})+D(\phi^k)}{k}=Q_2(\hat{\Phi}^n)+\frac{D(\phi^k)}{k}
\]
and taking the limit shows that $Q_2(\phi)=nQ_2(\hat{\Phi})$ (recall that $D(\phi^k)$ is uniformly bounded for every $k$). As $Q_2$ is Hofer $\left(3+\frac{1}{3A}\right)$-Lipschitz, we conclude that
\[
\norm{\phi}\geq \abs{\frac{n Q_2(\hat{\Phi})}{2}}\geq \frac{3A}{9A+1}\frac{3A-1}{2}n.
\]

\begin{rem}
We immediately see that our lower bound is an increasing function of $A$, which is zero in the limit case in which $A=\frac{1}{3}$: when the discs are displaceable, we cannot say anything this way about the Hofer norm of the diffeomorphisms. This is to be expected, as for $k=2$, by \cite{leR10} (see Remark following his Question 5) if the two discs are displaceable there exists a family of Hamiltonian diffeomorphisms, whose Hamiltonians are supported away from one of the two discs, with bounded Hofer norm but whose braid type may be arbitrarily linked.
\end{rem}

It is left to check that $Q_2$ vanishes on $S_0$. Using the (Hofer Lipschitz) and ($\scc^0-$continuity) properties of $Q_2$, it suffices to prove that $Q_2(S_0')=0$, where $S_0'$ is defined as the subgroup of Hamiltonian diffeomorphisms that induce the identity on a neighbourhood of $\partial\D\cup\partial D_1\cup\partial D_2$ (\cite{kha11}, Lemma 2). The argument is essentially the same as in \cite{kha11}, adapted to our quasimorphism $Q_2$. Firstly, to compute $Q_2$, by the first property in Theorem \ref{prop:mu_k}, this time we choose the obvious configuration $L_i=\partial D_i$. Now, if $\phi\in S_0'$, it can be decomposed into $\phi=\phi_{D_1}\circ \phi_{D_2}\circ \phi_P$, where $\phi_{D_i}$ is supported on $D_i$ and $\phi_P$ on $P=\D\setminus (D_1\cup D_2)$.

Since the boundaries of the $D_i$ are connected, $\phi_{D_i}\vert_{D_i}\in\Ham_c(D_i)\subseteq \Ham_c(\sphere^2\setminus (L_1\cup L_2))$, and by (Lagrangian Control) $Q_2(\phi_{D_i})=0$, thus it is left to show that $Q_2(\phi_P)=0$.

To prove this fact, as mentioned in \cite{kha11} (see also \cite{FarMar11}), it is possible to show that $\pi_0(\Symp_c(P))=\Z^3$, and that it is generated by three Dehn twists around the three boundary components. Below we show that if one applies an arbitrarily small Hofer deformation of $\phi$, taking place in $\Ham_c(\D)$, one can force $\phi_P\vert_P$ to be in the connected component of the identity of $\Symp_c(P)$. A Dehn twist around an embedded circle in $\D$ can be represented by a Hamiltonian whose Hofer norm is dependent on the diameter of the normal neighbourhood of the circle that we choose. In particular, we may represent the Dehn twists components in $[\phi_P\vert_P]\in\pi_0(\Symp_c(P))$ corresponding to curves near the boundaries of $D_i$ using $\psi\in\Ham_c(\D)$, with $\norm{\psi}<\varepsilon$; as for the component corresponding to a curve near $\partial \D$, it cannot appear in the decomposition since $\phi_P\vert_P$ is the restriction of an element in $S_0'$ (the presence of such factor forces the linking number to be $k$, where $k$ is the number of these Dehn twists counted with signs). As a result, $\psi^{-1}\circ\phi_P\vert_P$ is in the connected component of the identity in $\Symp_c(P)$, and it is enough to prove that $Q_2(\psi^{-1}\circ\phi_P\vert_P)=0$ by the Hofer Lipschitz property of $Q_2$.

$\psi^{-1}\circ\phi_P\vert_P$ is not necessarily in $\Ham_c(P)$, which would be enough to conclude; it may be however deformed to an element in $\Ham_c(P)$ without changing the value of $Q_2$. Let us consider the Flux homomorphism\[
\widetilde{\Symp}_c(P)\rightarrow H^1_c(P; \R)
\]
(definition and properties may be found in \cite{McDuffSalamon94})
and let us compute it on two Hamiltonian isotopies defined by $K_i$ where:
\begin{itemize}
    \item $K_i$ is supported on a small neighbourhood of $D_i$, equal to $1$ in a small neighbourhood of $L_i$ (here the $L_i$ constitute the premonotone Lagrangian configuration we had at the beginning of the section, it is not the pair of concentric circles on $\sphere^2$ we are using to compute the quasimorphism);
    \item $\supp(K_1)\cap \supp(K_2)=\emptyset$;
    \item $\supp(K_i)\cap \supp(\psi^{-1}\circ\phi_P\vert_P)=\emptyset$ for $i=1, 2$;
\end{itemize}

Since the two images $\Flux(\phi_{K_i}^t)$ are linearly independent (they are zero on one generator of $H_1(P, \partial P; \R)$ each) and $\dim(H^1_c(P; \R))=2$ there exist two real times $\xi_i\in\R$ such that
\[
\Flux(\phi^{\xi_1}_{K_1}\vert_P\circ \phi^{\xi_2}_{K_2}\vert_P\circ (\psi^{-1}\circ\phi_P\vert_P))=0
\]
and $\phi^{\xi_1}_{K_1}\circ \phi^{\xi_2}_{K_2}\circ (\psi^{-1}\circ\phi_P)$ may be represented by a Hamiltonian supported in $P$. Moreover,
\[
Q_2(\phi^{\xi_1}_{K_1}\circ \phi^{\xi_2}_{K_2}\circ (\psi^{-1}\circ\phi_P))=Q_2(\psi^{-1}\circ\phi_P).
\]

Indeed, since all the supports are disjoint, the three diffeomorphisms commute, therefore $Q_2$ is additive on them and \[
Q_2(\phi^{\xi_i}_{K_i})=\frac{1}{4}\xi_i-\frac{1}{4}\xi_i=0.
\]

Now we finish by remarking that 
\[
0=Q_2(\phi^{\xi_1}_{K_1}\circ \phi^{\xi_2}_{K_2}\circ (\psi^{-1}\circ\phi_P))=Q_2(\psi^{-1}\circ\phi_P)
\]
by (Support Control) property used on the premonotone Lagrangian configuration $L_1\times L_2$.

Summing up, given a diffeomorphism in $S_0$, it can be Hofer and $\scc^0$-deformed to an element in $S_0'$, changing by an arbitrarily small amount the value of $Q_2$. We then showed that in fact $Q_2(S_0')=0$, since any symplectic diffeomorphism in $S_0'$ is Hofer-close to one whose restriction to $P$ is symplectically isotopic to the identity of $P$. After this small perturbation, we applied a deformation without changing the value of $Q_2$, arriving to a Hamiltonian diffeomorphism with compact support in $P$, whose $Q_2$ value is necessarily 0 by an explicit computation.

\begin{rem}
It is possible to compare our estimates with the ones Khanevsky gives in \cite{kha11}. If $n$ is the rotation number of a non-displaceable disc in an annulus of area 1 under an isotopy between the identity and $\phi\in\Ham_c(\sphere^1\times(0, 1))$, he shows that
\[
\norm{\phi}\geq \frac{2A-1}{2}\abs{n}.
\]
To be able to compare the two estimates, one should remember that we always assumed our unit disc to have area 1, whereas Khanevsky only assumes the area of the annulus to be 1. In terms closer to our setup, this means that Khanevsky finds such estimate (up to a small change due to area normalisation) for compactly supported Hamiltonian diffeomorphisms of the disc $\D$, fixing a smaller disc contained in $\D$, and which may be represented by Hamiltonians supported away from another small disc (still contained in $\D$), disjoint from the first one. No further assumptions on the area of the latter small disc are made. 

We may see then that, while not requiring now the support of the Hamiltonian diffeomorphism to be contained in an annulus in $\D$, for our proof to work we do have to ask that the two small discs have the same area. Moreover, our estimate seems to be worse than the one given by Khanevsky.
\end{rem}
\printbibliography
\end{document}